\documentclass[11pt,reqno]{article}
\usepackage{amssymb, amsmath, amsthm, amsfonts, amscd, epsfig, subfig}
\usepackage[dvipsnames]{xcolor}
\usepackage[utf8]{inputenc}
\usepackage[english]{babel}
\usepackage{bm}
\usepackage{bbm}
\usepackage{graphicx, epsfig}
\usepackage{geometry,graphicx,pgfplots}

\usepackage[export]{adjustbox}
\usepackage[titletoc,toc,title]{appendix}

\usepackage{algorithm}
\usepackage{mathtools}

\usepackage{afterpage}
\usepackage{comment}

\newtheorem{theorem}{Theorem}[section]
\newtheorem{cor}[theorem]{Corollary}
\newtheorem{lemma}[theorem]{Lemma}                                                                                                                                                                                                                                                                             

\newtheorem{definition}{Definition}

\newtheorem{example}{Example}

\def\N{\mathbb{N}}

\newcommand{\Om}{\Omega}

\newcommand{\RR}{\mathbb{R}}

\newcommand{\p}{\partial}

\newcommand{\pd}[2]{\frac {\p #1}{\p #2}}
\newcommand{\ds}{\displaystyle}
\newcommand{\eqnref}[1]{(\ref {#1})}

\renewcommand{\qed}{\hfill $\Box$ \medskip}

\newcommand{\beq}{\begin{equation}}
\newcommand{\eeq}{\end{equation}}
\newcommand{\RN}[1]{
  \textup{\uppercase\expandafter{\romannumeral#1}}
}

\numberwithin{equation}{section}
\numberwithin{figure}{section}

\begin{document}
\title{Approximation of the first Steklov--Dirichlet eigenvalue on eccentric spherical shells in general dimensions \
}

\date{}

\author{
Jiho Hong\thanks{Department of Mathematics, The Chinese University of Hong Kong, Shatin, N.T., Hong Kong, P.R. China (jihohong@cuhk.edu.hk).}\and
Woojoo Lee\footnotemark[2]\thanks{Department of Mathematical Sciences, Korea Advanced Institute of Science and Technology, 291 Daehak-ro, Yuseong-gu, Daejeon 34141, Republic of Korea (woojoo.lee@kaist.ac.kr, mklim@kaist.ac.kr).}\and
Mikyoung Lim\footnotemark[2]
}

\maketitle

\begin{abstract}
We study the first Steklov--Dirichlet eigenvalue on eccentric spherical shells in $\RR^{n+2}$ with $n\geq 1$, imposing the Steklov condition on the outer boundary sphere, denoted by $\Gamma_S$, and the Dirichlet condition on the inner boundary sphere. The first eigenfunction admits a Fourier--Gegenbauer series expansion via the bispherical coordinates, where the Dirichlet-to-Neumann operator on $\Gamma_S$ can be recursively expressed in terms of the expansion coefficients \cite{Hong:2023:SDE:preprint}. In this paper, we develop a finite section approach for the Dirichlet-to-Neumann operator to approximate the first Steklov--Dirichlet eigenvalue on eccentric spherical shells.
We prove the exponential convergence of this approach by using the variational characterization of the first eigenvalue.
Furthermore, based on the convergence result, we propose a numerical computation scheme as an extension of the two-dimensional result in \cite{Hong:2022:FSD} to general dimensions. We provide numerical examples of the first Steklov--Dirichlet eigenvalue on eccentric spherical shells with various geometric configurations.

\end{abstract}



\section{Introduction}
We consider the Steklov--Dirichlet eigenvalue problem for a smooth domain $\Om\subset\RR^d$ with two boundary components $\Gamma_D$ and $\Gamma_S$:
\begin{align} \label{eqn:Steklov-Dirichlet}
\begin{cases} 
    \ds \Delta u = 0   &\ds\text{in   } \Omega,\\
  \ds   u = 0 & \ds\text{on   }  \Gamma_D,\\
  \ds   \frac{\partial u}{\partial n} = \sigma u &\ds\text{on   } \Gamma_S
\end{cases}
\end{align}
with the unit outward normal vector $n$ to $\p\Om$. 
A real constant $\sigma$ is called a Steklov--Dirichlet eigenvalue if there exists a non-trivial solution $u$, the corresponding eigenfunction, to \eqnref{eqn:Steklov-Dirichlet}. 
For the instance $\Gamma_D=\emptyset$, the eigenvalue problem \eqnref{eqn:Steklov-Dirichlet} degenerates to the classical Steklov eigenvalue problem, for which we refer to \cite{stekloff:1902:FPM,Girouard:2017:SGS,Colbois:2022:SRD}. 
Assuming $\Gamma_D\neq\emptyset$, \eqnref{eqn:Steklov-Dirichlet} admits discrete eigenvalues (see \cite{Agranovich:2006:MPS}), namely,  
$$ 0<\sigma_1(\Omega)\le \sigma_2(\Omega) \le \cdots \rightarrow \infty.$$ 
The first Steklov--Dirichlet eigenvalue admits the variational characterization \cite{Bandle:1980:IIA}:
\begin{align} \label{variational characterization}
\sigma_1(\Om)=\inf_{v\in H^1_\diamond(\Om)\setminus\{0\}}\frac{\|\nabla v\|_{L^2(\Om)}^2}{\|v\|_{L^2(\Gamma_S)}^2}
\end{align}
with $H^1_\diamond(\Om):=\{v\in H^1(\Om)\,:\, v=0\mbox{ on }\Gamma_D\}$. 
In addition, the following variational characterization holds (see, for example, \cite[Eqn. (2.7)]{Colbois:2024:SRD}):
\begin{align}\label{vari:second}
\sigma_2(\Omega) = \inf_{E\in\mathcal{E}} \sup_{v\in E\setminus\{0\}} \frac{\|\nabla v\|_{L^2(\Om)}^2}{\|v\|_{L^2(\Gamma_S)}^2},
\end{align}
where $\mathcal{E}$ is the set of all two dimensional subspaces of $H^1_\diamond(\Om)$.
The Steklov--Dirichlet eigenvalue problem \eqnref{eqn:Steklov-Dirichlet} is equivalent to the eigenvalue problem of the  Dirichlet-to-Neumann operator $\mathcal{L}$ defined by 
\begin{align}\label{def:oper:L}
\mathcal{L} :    \hat{u} &\mapsto \frac{\partial {u}}{\partial \nu}\Big|_{\Gamma_S}\quad\mbox{on }C^\infty(\Gamma_S)
\end{align}
with the solution $u$ to the problem
\begin{align*}
\begin{cases}
\ds \Delta {u} = 0&   \text{in   } \Omega,  \\
\ds {u} = 0& \text{on   } \Gamma_D,\\
\ds {u} = \hat{u} &\text{on   } \Gamma_S.
\end{cases}
\end{align*}
The operator $\mathcal{L}$ is positive-definite and self-adjoint with respect to the $L^2$ inner product \cite{Agranovich:2006:MPS}. 

The Steklov--Dirichlet eigenvalue problems are related to various other problems. For instance, the vibration modes of a partially free membrane, fixed along the inner boundary with no mass on the interior, can be described by Steklov--Dirichlet eigenfunctions \cite{Hersch:1968:EPI}. The eigenvalue problem shares connections with the Laplace eigenvalue problems \cite{Arrieta:2008:FRL, Lamberti:2015:VSE} and the stationary heat distribution \cite{Banuelos:2010:EIM, Kuznetsov:2014:LVA}. 
In addition, the Steklov--Neumann eigenvalue problem, which is the problem \eqref{eqn:Steklov-Dirichlet} with the zero Neumann condition instead of the zero Dirichlet condition on $\Gamma_D$, has relevance to the sloshing problem in hydrodynamics \cite{Kulczycki:2009:HST}. An optimization approach for the Steklov--Neumann eigenvalues was studied in \cite{Ammari:2020:OSN}. We refer to \cite{Banuelos:2010:EIM}  for a comparison of the Steklov--Dirichlet and Steklov--Neumann eigenvalues.

The geometric dependence of the first Steklov--Dirichlet eigenvalue has been intensively studied. 
 In 1968, Hersch and Payne obtained bounds on the first eigenvalue on bounded doubly connected domains in $\mathbb{R}^2$ \cite{Hersch:1968:EPI}. For planar domains, Dittmar derived isoperimetric inequalities  \cite{Dittmar:1998:IIS}, and Dittmar and Solynin obtained a lower bound for doubly connected domains \cite{Dittmar:2003:MSE, Dittmar:2005:EPC}.  See also \cite{Paoli:2021:SRS, Micetti:2022:SDS} for  spectral stability and \cite{Hassannezhad:2020:EBM} for the Riesz mean estimates of the mixed Steklov eigenvalues.

For the instance in which $\Om$ is an eccentric spherical shell, which is the main subject of this paper, much attention has been attracted to establishing the behavior of the first Steklov--Dirichlet eigenvalue depending on the distance $t$ between the two centers of the boundary spheres of the shell (see Figure \ref{fig:geometry}). For simplicity, we denote by $\sigma_1^t$ the first Steklov--Dirichlet eigenvalue on the eccentric shell. Santhanam and Verma proved that $\sigma_1^t$ attains the maximum at $t=0$ in $(n+2)$-dimensions with $n\geq 1$ \cite{Verma:2020:EPL}, and Seo and Ftouhi independently showed the maximality in $\RR^2$ \cite{Seo:2021:SOP, Ftouhi:2022:WPS};  this maximality result was generalized to two-point homogeneous spaces \cite{Seo:2021:SOP} and general domains in Euclidean spaces \cite{Gavitone:2022:IIF}. Hong, Lim and Seo verified differentiability for $\sigma_1^t$ with respect to $t$ and obtained its shape derivative \cite{Hong:2022:FSD}.  Also, the shape derivative and the dependence of the first eigenvalue on $t$ have been investigated for other Laplacian eigenvalues problem. We refer the reader to \cite{Ramm:1998:IME, Chorwadwala:2013:TFL, Anoop:2018:SMF, Rane:2019:FDE} for the Dirichlet Laplacian problems, to \cite{Chorwadwala:2015:EOP, Anoop:2018:SMF} for the Dirichlet $p$-Laplacian problems, to \cite{Djitte:2021:FHF,Djitte:2023:NSF} for the Dirichlet fractional Laplacian problems, and to \cite{Anoop:2021:SVR} for the Zaremba problem.

In this paper, we present an approximation method for $\sigma_1^t$ by generalizing the result in two dimensions \cite{Hong:2022:FSD} to arbitrary higher-order dimensions. 
It is noteworthy that the convergence of the approximation of $\sigma_1^t$ and the corresponding eigenfunction is established here (see Theorem \ref{prop:computation:validation:eq} and Theorem \ref{thm:eigenfunction}), whereas they were not in \cite{Hong:2022:FSD}.
To describe the result in detail we specify the geometric configuration. 
Let $\Om$ be an eccentric spherical shell in $\RR^{n+2}$, $n\geq 1$, where 
the zero Dirichlet condition and the robin boundary condition are assigned on the inner and outer boundaries of $\Om$, respectively. In other words, we consider the eigenvalue problem \eqnref{eqn:Steklov-Dirichlet} for the domain
\beq\label{spheri_shell}
\Om=B_2\setminus\overline{B_1^t}\quad\mbox{with}\quad \Gamma_D=\p B_1^t,\ \Gamma_S=\p B_2,
\eeq
where $B_1^t$, $B_2$ are balls satisfying $\overline{B_1^t}\subset B_2$ and $t$ is the distance between the centers of the inner and outer boundary spheres of $\Om$ (see Figure \ref{fig:geometry}). We denote by $\sigma_1^t$ the first 
Steklov--Dirichlet eigenvalue as above and by $u_1^t$ the corresponding eigenfunction. 
The exact value for the concentric case (i.e., $t=0$) is well known as
\begin{align}\label{sigma1:0}
\sigma_1^0 = \frac{nr_1^{n}}{r_2\big({r_2}^{n}- r_1^{n}\big)},
\end{align}
which is the maximal value of $\sigma_1^t$ over $t$ \cite{Verma:2020:EPL}.

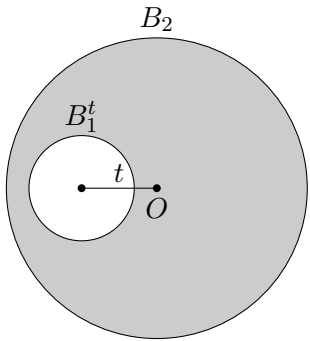
\begin{figure}[h]
    \centering
        \begin{tikzpicture}
\coordinate  (C) at (0,0);
\coordinate (D) at (2,0);
\coordinate (X) at (-1,0);
\fill[gray!40,even odd rule] (X) circle (0.7) (C) circle (2);
\draw (C) circle(2cm);
\draw (X) circle (0.7cm);

\draw (X) --(C);
\node at (-0.5,0.2) {$t$};

\node at (-1,0.95){$B_1^t$};
\node at (0,2.25){$B_2$};
\node[below] at (0,0){$O$};

\foreach \point in {X,C}
	\fill [black] (\point) circle (1.5pt);

	\end{tikzpicture}
    \caption{An eccentric spherical shell $\Om=B_2\setminus\overline{B_1^t}$ with the distance $t$ between the centers of  the two boundary spheres.}  \label{fig:geometry}
\end{figure}

 For the spherical shell $\Om=B_2\setminus\overline{B_1^t}$ in $\RR^d$ with $d\geq 2$,  the first eigenvalue $\sigma_1^t$ is simple, i.e., 
 \beq\label{first:simple}
 \sigma_1^t<\sigma_2^t,
 \eeq
 and the corresponding eigenfunction $u_1^t$ does not change the sign in $\Om$ \cite{Hong:2022:FSD}.
Assuming $d=n+2$ with $n\geq 1$ and appropriately rotating and translating $\Om$, one can express the boundary values of the first eigenfuction as
\beq\label{eq:final series expansion of u_1^t}
\begin{aligned}
\ds u_1^t\Big|_{\p B_2}&=(\cosh\xi_2-\cos\theta)^{\frac{n}{2}}\,\sum_{m=0}^\infty \widetilde{C}_m \,G_m^{\left(n/2\right)}(\cos\theta),\\[1mm]
\ds\frac{\partial u_1^t}{\partial {n}}\Big|_{\partial B_2}
&=-\frac{(\cosh \xi_2-\cos \theta)^{\frac{n}{2}}}{\alpha}
 \sum_{m=0}^\infty \Big(\frac{n\sinh\xi_2}{2}\widetilde{C}_m-\cosh\xi_2\big(m+\frac{n}{2}\big)c_m^2\widetilde{C}_m\\
& \hskip 4cm+\frac{m}{2}c_{m-1}^2\tilde{C}_{m-1}
+\frac{m+n}{2}c_{m+1}^2\widetilde{C}_{m+1}\Big)G_m^{\left(n/2\right)}(\cos \theta)
\end{aligned}
\eeq
with some constant coefficients $\widetilde{C}_m$ \cite{Hong:2023:SDE:preprint}. Here, $(\xi,\theta,\varphi_1,\dots, \varphi_{n})$ is the bispherical coordinate system and $G_m^{\left(n/2\right)}$ indicates the Gegenbauer polynomials (see section \ref{sec:sec2} for details).  We remind the reader that the eigenvalue value problem for $\mathcal{L}$ in \eqnref{def:oper:L} is equivalent to the Steklov--Dirichlet eigenvalue problem.
The Steklov condition on $\Gamma_S(=\p B_2)$ leads to the recursive relations for the coefficients $\widetilde{C}_m$. 
Using this recursive relation, an asymptotic lower bound is obtained \cite{Hong:2023:SDE:preprint} (see also \cite{Hong:2022:FSD} for the results in $\RR^2$):
\beq\label{ineq:sdeig:lowerbound}
\liminf_{t\to (r_2-r_1)^-}\sigma_1^t\ge\frac{(n+1)r_1-nr_2}{2r_2(r_2-r_1)}.
\eeq

In the present paper, we apply the finite section method (see, for instance, \cite{Bottcher:2000:CNA,Bottcher:2001:AAN}) and represent the operator $\mathcal{L}$ by a symmetric tridiagonal matrix, similar to the approach in  \cite{Hong:2022:FSD}. 
We take the finite section operator $Q_N\mathcal{L}Q_N$ for the  Dirichlet-to-Neumann operator $\mathcal{L}$ in  \eqnref{def:oper:L} with an orthogonal projection $Q_N$, where $Q_N\mathcal{L}Q_N$ is identical to a finite dimensional matrix we name $\mathbb{L}_N$. We denote by $\sigma_{1,N}^t$ the smallest eigenvalue of $\mathbb{L}_N$ and define $u_{1,N}^t$ using the first eigenvectors of $\mathbb{L}_N$. Our main theorems are the following. We provide the proofs in subsection \ref{subsec:proof:maintheorem}. 

\begin{theorem}\label{prop:computation:validation:eq}
Let $m\in\mathbb{N}$ and $u_{1,m}^t$ be given by  Definition \ref{def:trun:u} in subsection \ref{subsec:finitesection}. We have
\beq\label{bounds:stronger}\lim_{N\to\infty}\sigma_{1,N}^t= \sigma_1^t
\le \frac{\|\nabla u_{1,m}^t \|_{L^2(\Om)}^2}{\|u_{1,m}^t \|_{L^2(\Gamma_S)}^2}\le \sigma_{1,m}^t.\eeq
\end{theorem}

\begin{theorem}\label{thm:sigma:decay}
Let $\Om$ be given by \eqnref{spheri_shell}. 
For some $\delta,C,N_0>0$ independent of $N$, it holds that
\beq\label{eqn:sigma:decay:pos}
0<\sigma_{1,N}^t-\sigma_1^t\leq C e^{-N\delta}\quad\mbox{for all }N\geq N_0.
\eeq
\end{theorem}

\begin{theorem}\label{thm:eigenfunction}
Let $\Om$ be given by \eqnref{spheri_shell}. Let $u_{1,m}^t$ be an eigenfunction corresponding to $\sigma_{1,m}^t$ (see \eqnref{def:u1N}). We normalize $u_1^t$ and $u_{1,m}^t$ so that $\|\nabla u_1^t\|_{L^2(\Om)}=\|\nabla u_{1,m}^t\|_{L^2(\Om)}=1$ and $\langle\nabla u_{1,m}^t,\, \nabla u_1^t\rangle_{L^2(\Om)}\ge0$. 
	For some $\delta,C,M>0$ independent of $m$, it holds that
	\begin{equation}\label{eq:converge:eigftn}
		\|u_{1,m}^t- u_{1}^t\|_{H^1(\Om)}\leq C e^{-m\delta}\quad\mbox{for all }m\geq M.	
	\end{equation}
\end{theorem}

The finite dimensional matrix $\mathbb{L}_N$ is symmetric, positive definite and tridiagonal, and each entry of $\mathbb{L}_N$ is explicitly determined in terms of $r_1$, $r_2$, $t$ and $n$. Therefore, one can easily compute the first eigenvalue of $\mathbb{L}_N$ and estimate $\sigma_1^t$ in $\RR^{n+2}$ for arbitrary $n\geq 1$. This eigenvalue computation does not require mesh generation unlike the finite difference method \cite{Kuttler:1970:FDA,Froese:2018:MFD} or the finite element method \cite{Fix:1973:EAF,Boffi:2010:FEA,Borthagaray:2018:FEA}. In addition to the computational ease, it is robust in that its exponential convergence is guaranteed by Theorems \ref{prop:computation:validation:eq} and \ref{thm:sigma:decay}.

The rest of this paper is organized as follows. In section \ref{sec:sec2}, we represent the first eigenfunction as a series expansion by using the bispherical coordinates. Section \ref{section:computation} introduces the finite section method to approximate the first eigenvalue and provides its convergence. In section \ref{sec:numerical}, we propose a numerical scheme to compute $\sigma_1^t$ and demonstrate numerical examples with various geometric configurations. We conclude the paper with brief discussions in section \ref{sec:conclusion}.

\section{The first eigenfunction $u_1^t$ in bishperical coordinates}\label{sec:sec2}
\subsection{Bispherical coordinates} \label{sec: the first: bispherical coordinates}
For a given fixed $\alpha>0$. 
The bispherical coordinates for $\mathbf{x}=(x_1,x_2,x_3)\in\RR^3$ are defined by 
\begin{align}\label{Bi:3D}
x_1=\frac{\alpha\sinh\xi}{\cosh\xi-\cos\theta}, \quad
x_2=\frac{\alpha\sin\theta\cos\varphi_1}{\cosh\xi-\cos\theta},\quad
x_3=\frac{\alpha\sin\theta\sin\varphi_1}{\cosh\xi-\cos\theta}.
\end{align}
Here, we denote by ${B}^3_j(\xi,\theta,\varphi_1)$ the coordinate functions of $x_j$.
We then recursively define the  bispherical coordinates for ${x}=(x_1,\dots,x_{n+2})\in\RR^{n+2}$, $n\geq 2$, by
\begin{align}\notag
x_j=&B^{n+2}_{j}\left(\xi,\theta,\varphi_1,\dots \varphi_{n}\right)\\ \label{bihyperspherical coordinates}
:=&
\begin{cases}
\ds B^{n+1}_{j}\left(\xi,\theta,\varphi_1,\dots \varphi_{n-1}\right)\quad& \mbox{for }j=1,\dots,n,\\[1mm]
\ds B^{n+1}_{n+1}\left(\xi,\theta,\varphi_1,\dots \varphi_{n-1}\right) \cos\varphi_n\quad&\mbox{for }j=n+1,\\[1mm]
\ds B^{n+1}_{n+1}\left(\xi,\theta,\varphi_1,\dots \varphi_{n-1}\right)\sin\varphi_n\quad&\mbox{for }j=n+2,
\end{cases}
\end{align}
for $(\xi,\theta,\varphi_1,\dots \varphi_{n})\in \RR\times[0,\pi]^{n}\times[0,2\pi)$.
Set ${y}=(y_1,y_2,y_3,\dots,y_{n+2})=(\xi,\theta,\varphi_1,\dots,\varphi_n)$ and define
$$g_{ij}:=\left\langle \pd{{x}}{y_i},\, \pd{{x}}{y_j}\right\rangle_{\RR^{n+2}}\quad \mbox{for }i,j=1,\dots,n+2.$$
We have
\beq\label{Jacobian}
\sqrt{\det (g_{ij})}=\frac{\alpha^{n+2}\sin^{n}\theta\sin^{n-1}\varphi_1\cdots\sin^2\varphi_{n-2}\sin\varphi_{n-1}}{(\cosh\xi-\cos\theta)^{n+2}}.
\eeq

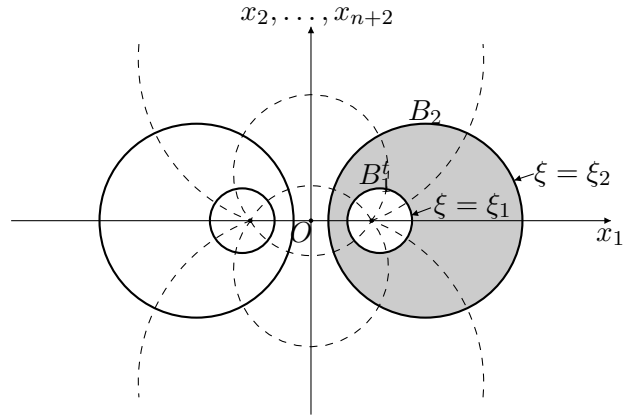
\begin{figure}[b!]
\begin{center}

\begin{tikzpicture}[scale=0.43]

\coordinate  (C) at (3.537742925, 0);
\coordinate (X) at (2.125, 0);
\fill[gray!40,even odd rule] (X) circle (1) (C) circle (3);

\node at (2,1.37){$B_1^t$};
\node at (3.537742925,3.35){$B_2$};

\draw[thick] (3.537742925, 0) circle (3);
\draw[thick] (2.125, 0) circle (1);
\draw[thick] (-3.537742925, 0) circle (3);
\draw[thick] (-2.125, 0) circle (1);

\draw (0, -6.0) -- (0, 6.0);
\draw (-9.27, 0) -- (9.27, 0);

\fill (0, 6.0) -- (-0.1, 5.8) -- (0.1, 5.8);
\fill (9.27, 0) -- (9.07, 0.1) -- (9.07, -0.1);
\draw (9.27, -0.5) node {$x_1$};
\draw (0.2, 6.3) node {$x_2,\dots, x_{n+2}$};
\draw (-0.3, -0.35) node {$O$};

\draw[dashed, domain=30:150] plot ({2.1650635*cos(\x)},{-1.08253175+2.1650635*sin(\x)});
\draw[dashed, domain=-150:-30] plot ({2.1650635*cos(\x)},{1.08253175+2.1650635*sin(\x)});
\draw[dashed, domain=-40:220] plot ({0+2.40117*cos(\x)}, {1.5+2.40117*sin(\x)});
\draw[dashed, domain=140:400] plot ({0+2.40117*cos(\x)}, {-1.5+2.40117*sin(\x)});
\draw [dashed, domain=-5:70] plot ({5.34*cos(\x)}, {5.34*sin(\x)-5});
\draw [dashed, domain=175:250] plot ({5.34*cos(\x)}, {5.34*sin(\x)+5});
\draw [dashed, domain=110:185] plot ({5.34*cos(\x)}, {5.34*sin(\x)-5});
\draw [dashed, domain=290:365] plot ({5.34*cos(\x)}, {5.34*sin(\x)+5});

\fill (1.875, 0) circle (0.07);
\fill (-1.875, 0) circle (0.07);
\fill (0, 0) circle (0.07);

\draw (5, 0.46) node {$\xi=\xi_1$};
\draw (8.1, 1.45) node {$\xi=\xi_2$};

\draw (3.1098, 0.173648) -- (3.7098, 0.4);
\draw (6.29006, 1.25) -- (6.89006, 1.45);

\fill (3.1098, 0.173648) -- (3.1098+0.156693501419659, 0.173648+0.159521618011000) -- (3.1098+0.221633395260595, 0.173648-0.0296418303291268);
\fill (6.29006, 1.25) -- (6.29006+0.156693501419659, 1.25+0.159521618011000) -- (6.29006+0.221633395260595, 1.25-0.0296418303291268);

\end{tikzpicture}

\end{center}
\vskip -4mm
\caption{\label{fig:coord:3d}$\xi$ (thick curves) and $\theta$ (dashed curves) level surfaces of the bispherical coordinate system in $\mathbb{R}^{n+2}$. We choose $\alpha$ by \eqnref{def:alpha} so that $\p B_1^t$ and $\p B_2$ are $\xi$-level curves. }
\end{figure}

We investigate the first Steklov--Dirichlet eigenvalue on eccentric spherical shells $\Om=B_2\setminus\overline{B_1^t}$ in general dimensions where $r_1$ (resp. $r_2$) denotes the radius of the inner (resp. outer) boundary sphere and $t$ is the distance between the centers of the inner and outer boundary spheres. 

Set
\begin{align}\label{def:alpha}
\alpha&=\frac{1}{2t}\sqrt{\left((r_2+r_1)^2-t^2\right)\left((r_2-r_1)^2-t^2\right)},\\ \notag
\xi_j&=\ln\big(({\alpha}/{r_j})+\sqrt{({\alpha}/{r_j})^2+1}\big),\ j=1,2.
\end{align}
Note that $\xi_1>\xi_2>0$. 
By rotating and translating $\Om$ with an appropriately chosen $t_0$, we have (see Figure \ref{fig:coord:3d})
\beq\label{eqn:B1B2:2}
B_1^t=t_0e_1+B(-te_1,r_1),\quad B_2=t_0e_1+B(0,r_2)
\eeq
and
\begin{gather*}
\p B_1^t=\{\xi=\xi_1\},\ \p B_2=\{\xi=\xi_2\}.
\end{gather*}
Here,  $e_1=(1,0,\dots,0)$ and $B(x,r)$ indicates the ball centered at $x$ with the radius $r$.
For a function $u$, it holds that
\beq\label{normal:u:3d}
\left.\pd{u}{{n}}\right|_{\p B_2}=\left.-\frac{1}{h({\xi},\theta)}\pd{u}{\xi}\right|_{\xi=\xi_2}\quad\mbox{with } h(\xi,\theta) = \frac{\alpha}{\cosh\xi-\cos\theta}.
\eeq
The surface integral on $\p B_2$ admits the relation that
\beq\label{surint:B2}
\begin{aligned}
&\int_{\p B_2}\cdot\ dS
 = \int_0^\pi\cdots\int_0^\pi\int_{0}^{2\pi}\cdot\ \frac{\alpha^{n+1}\sin^n\theta  }{(\cosh\xi_2-\cos\theta)^{n+1}}\prod_{j=1}^n\left(\sin^{n-j}\varphi_j\right)d\theta d\varphi_1\dots d\varphi_n.
\end{aligned}
\eeq

\subsection{Series expressions of the first eigenfunction} \label{sec:first eigenfunction}

For a given fixed $\lambda\in(0,\infty)$, the Gegenbauer polynomials (or, ultraspherical polynomials) $G_m^{(\lambda)}(s)$ are given by the generating relation (see, for instance, (4.7.23) in \cite{Szego:1975:OP})
$$(1-2st+t^2)^{-\lambda}=\sum_{m=0}^\infty G_{m}^{(\lambda)}(s)\,t^m\quad\mbox{for }s\in(-1,1),\ t\in[-1,1].$$
For instance, the lowest order polynomials are $G_0^{(\lambda)}(s)=1$ and $G_1^{\lambda}(s)=2s\lambda$. Higher-order terms can be easily obtained by the recurrence relation (see (4.7.17) in \cite{Szego:1975:OP}): for all $m\geq 2$, 
    \begin{equation}\label{Gegen:recur}
\begin{gathered}
    mG_m^{(\lambda)}(s)-2(m+\lambda-1)sG_{m-1}^{(\lambda)}(s)+(m+2\lambda-2)G_{m-2}^{(\lambda)}(s)=0.
    \end{gathered}
    \end{equation}
    The Gegenbauer polynomials $G_{m}^{(\lambda)}(s)$, ${m\ge 0}$, form a complete orthogonal basis for the weighted $L^2$ space $L^2\left([-1,1];(1-s^2)^{\lambda-1/2}\,ds\right)$; we refer the reader, for instance, to \cite[Corollary IV 2.17]{Stein:1971:IFA}.

The Gegenbauer polynomial of $\cos\theta$ satisfies the expansion
\begin{equation} \label{Gegen:cos}
G_m^{(\lambda)}(\cos\theta) = \sum_{k=0}^m \frac{\lambda^{(k)}}{k!} \frac{\lambda^{(m-k)}}{(m-k)!} \cos((m-2k)\theta),
\end{equation}
which can be derived from the generating relation (see (3.15.13) in \cite{Bateman:1953:HTF}).

As shown in \cite[Theorem 7.33.1 and (4.7.3)]{Szego:1975:OP} (see also \cite{Hong:2023:SDE:preprint}), it holds that
   \begin{gather}\label{gegen_maximum}
   \left|G_m^{(\lambda)}(s)\right|\leq Cm^k\quad\mbox{for all }s\in[-1,1],\\
\label{Gegen:norm}
\|G_m^{(\lambda)}\|^2_{\lambda-\frac{1}{2}}:=\int_{-1}^1 \left(G_m^{(\lambda)}(s)\right)^2 (1-s^2)^{\lambda -\frac{1}{2}}\,ds
\le Cm^{k}
\end{gather}
for some constants $C=C(\lambda)>0$ and $k=k(\lambda)>0$.

The first eigenfunction $u_1^t$ on eccentric spherical shells admits the series expansion to  the Gegenbauer polynomials with $\lambda=\frac{n}{2}$ in terms of the bispherical coordinates as follows. 
\begin{lemma}[\cite{Hong:2023:SDE:preprint}]\label{prop:u_1^t seriesexp}
Fix an arbitrary $n\geq 1$. Let $\Om=B_2\setminus\overline{B_1^t}$ be an eccentric spherical shell in $\RR^{n+2}$ given by \eqnref{eqn:B1B2:2} and $u_1^t(x)$ be the first Steklov--Dirichlet eigenfunction for $\Om$. 
In terms of the bispherical coordinates $(\xi,\theta,\varphi_1,\dots, \varphi_{n})$ described in Section \ref{sec: the first: bispherical coordinates}, $u_1^t$ admits the expansion
\begin{align} \label{u_1^t seriesexp}
u_1^t\left({x}\right)=(\cosh\xi-\cos\theta)^{\frac{n}{2}}\,\sum_{m={0}}^\infty & C_m\left( e^{(m+\frac{n}{2})(2\xi_1-\xi)}-e^{(m+\frac{n}{2})\xi} \right)G_{m}^{\left(n/2\right)}(\cos\theta)
\end{align} 
with some constant coefficients $C_m$. 
\end{lemma}

For simplicity, we set
\begin{align}
\notag\ds\widetilde{C}_m&:=C_m\big( e^{(m+\frac{n}{2})(2\xi_1-\xi_2)}-e^{(m+\frac{n}{2})\xi_2}\big),\\ \label{def:cm}
\ds c_m&:=\left(\tanh\big(\big(m+\frac{n}{2}\big)(\xi_1-\xi_2)\big)\right)^{-1/2}\neq 0\quad\mbox{for each }m\geq 0. 
\end{align}
One can show the convergence of \eqnref{u_1^t seriesexp} by the following relation (see \cite{Hong:2023:SDE:preprint}):
\beq\label{Cm_tilde:uppperbound} 
\mbox{for some }\delta>0,\quad \left|\widetilde{C}_m\right|
=O\left(e^{-(m+\frac{n}{2})\frac{\delta}{2}}\right)\quad\mbox{as }m\rightarrow\infty. 
\eeq

As the right-hand side in \eqnref{u_1^t seriesexp} satisfies the conditions $\Delta u=0$ in 
$\Om$ and $u=0$ on $\p B_1$, it is enough to only consider the Steklov boundary condition $\frac{\partial u_1^t}{\partial {n}}=\sigma_1^tu_1^t$ on $\partial B_2$ to find the first eigenvalue $\sigma_1^t$.
We obtain \eqnref{eq:final series expansion of u_1^t} by \eqnref{u_1^t seriesexp} and \eqnref{Cm_tilde:uppperbound}.
By \eqnref{eq:final series expansion of u_1^t} and the Steklov boundary condition in \eqnref{eqn:Steklov-Dirichlet}, we have the following relations:
\begin{lemma}[\cite{Hong:2023:SDE:preprint}] \label{lemm: recursive}
	We have 
	\beq\notag
	\begin{aligned}
		&\big(-2\alpha\sigma_1^t-n\sinh\xi_2+nc_0^2\cosh\xi_2\big)\widetilde{C}_0 - nc_1^2\widetilde{C}_1 = 0,\\
		&\big(-2\alpha\sigma_1^t-n\sinh\xi_2+(2m+n)c_m^2\cosh\xi_2\big)\widetilde{C}_m - mc_{m-1}^2\widetilde{C}_{m-1} -  (m+n)c_{m+1}^2\widetilde{C}_{m+1}=0,\quad m\geq 1. 
	\end{aligned}
	\eeq
\end{lemma}

\section{Approximation of $\sigma_1^t$ by the finite section method}\label{section:computation}

\subsection{Dirichlet-to-Neumann operator}\label{subsec:DtN}
Recall that we consider the domain $\Om=B_2\setminus\overline{B_1^t}\subset\RR^{n+2}$ with $\Gamma_D=\p B_1^t=\{\xi=\xi_1\}$ and $\Gamma_S=\p B_2=\{\xi=\xi_2\}$. 
Let $(\xi,\theta)$ be the first two components in the bispherical coordinates and set $s=\cos\theta$. 
Define
\begin{align}\notag
\tilde{g}_{-1}(s)&=0,\\
\tilde{g}_k(s)&=(\cosh\xi_2-s)^{\frac{n}{2}}\, G_k^{(n/2)}(s),\quad k\geq 0,
\end{align}
and
$$g_k(s):=\Big(\prod_{j=1}^k \frac{\sqrt{j}}{\sqrt{j+n-1}}\Big)\frac{1}{c_k}\,\tilde{g}_k(s),\quad k\geq 0.$$
Then $\{\tilde{g}_k(s)\}_{k\geq 0}$ is a complete orthogonal basis for $L^2\left([-1,1]; (1-s^2)^{n/2-1/2}(\cosh\xi_2-s)^{-n}\,ds\right)$.

One can rewrite \eqnref{eq:final series expansion of u_1^t} as
\beq\label{eq:compact series expansion of normal derivative of u_1^t}
\begin{cases}
\ds u_1^t\big|_{\p B_2}(x)=\sum_{k=0}^\infty \widetilde{C}_k \,\tilde{g}_k(s),\\[1mm]
\ds\frac{\partial u_1^t}{\partial {n}}\Big|_{\partial B_2}(x)
=\frac{1}{2\alpha}\sum_{k=0}^\infty  \widetilde{C}_k\Big[\!-(k+n-1)c_k^2\, \tilde{g}_{k-1}(s)\\
\hskip 2cm  + \left((2k+n)c_k^2\cosh \xi_2 - n \sinh \xi_2\right) \tilde{g}_k(s) 
- (k+1)c_k^2\, \tilde{g}_{k+1}(s)\Big].
\end{cases}
\eeq
The right-hand sides in these equations belong to $L^2\left([-1,1]; (1-s^2)^{n/2-1/2}(\cosh\xi_2-s)^{-n}\,ds\right)$ by \eqnref{Cm_tilde:uppperbound} and \eqnref{Gegen:norm}.
Since \eqnref{eq:final series expansion of u_1^t} is derived for $u_1^t$ satisfying $\Delta u=0$ in $\Om$ and $u=0$ on $\Gamma_D=\p B_1^t$, \eqnref{eq:compact series expansion of normal derivative of u_1^t} implies 
\begin{align}\label{eq:L:mapping} 
\notag &\mathcal{L}\left[ \tilde{g}_k(s)\right] =  \frac{1}{2\alpha} \Big[\!-(k+n-1)c_k^2\,\tilde{g}_{k-1}(s) \\
&\qquad \qquad \qquad  + \left((2k+n)c_k^2\cosh \xi_2 - n \sinh \xi_2\right) \tilde{g}_k(s) - (k+1)c_k^2\, \tilde{g}_{k+1}(s)\Big]
\end{align}
and, thus,
\beq \label{def:dm:wm}
\begin{gathered}
\mathcal{L}\left[ {g}_k(s)\right] =  
\frac{1}{2\alpha}d_k g_{k}(s)+\frac{1}{2\alpha}\left(-w_k c_{k-1}c_k g_{k-1}(s)
-w_{k+1}c_{k}c_{k+1}g_{k+1}(s)\right),\\
\notag
d_k=(n+2k)c_k^2\cosh \xi_2 -n\sinh\xi_2, \quad w_k = \sqrt{(k+n-1)k}.
\end{gathered}
\eeq

\subsection{The first eigenvalue of the finite section of $\mathcal{L}$}\label{subsec:finitesection}
We define the finite dimensional space
$$H_N:=\operatorname{span}\{{g}_0(s),{g}_1(s),\cdots,{g}_{N-1}(s)\}\quad\mbox{for each }N=1,2,\dots. $$
Set $Q_N$ to be the orthogonal projection onto $H_N$.
We define the inner product $(\cdot,\cdot)$ on $H_n$ by
\beq\label{inner:def:Hn}
(g_j,g_k)=\delta_{jk}\quad\mbox{for }j,k=0,1\dots,N-1,
\eeq
$\delta_{ij}$ being the Kronecker delta. 
By \eqnref{def:dm:wm}, one can identify the finite section $Q_N\mathcal{L}Q_N$ of $\mathcal{L}$ with respect to $\{g_0(s),\cdots,g_{N-1}(s)\}$ by the symmetric tridiagonal matrix 
$\mathbb{L}_N$ given by
\begin{gather} \label{def:M}
\ds \mathbb{L}_N=\frac{1}{2\alpha}\left(\mbox{diag}(d_0,\cdots,d_{N-1}) -\mathbb{T}_N\right)
\end{gather}
with
\beq\notag
\mathbb{T}_N
=\left(\begin{array}{ccccc}
0& w_1c_0c_1 &&&\\[1mm]
w_1c_0c_1 &0&w_2c_1c_2&&\\[1mm]
&w_2c_1c_2&0&\ddots&\\[1mm]
&&\ddots&\ddots&w_{N-1}c_{N-2}c_{N-1}\\[1mm]
&&&w_{N-1}c_{N-2}c_{N-1}&0
\end{array}\right).
\eeq

\smallskip
\smallskip
\begin{lemma}
The matrix $\mathbb{L}_N$ is symmetric positive definite.
\end{lemma}
\begin{proof}
 It is sufficient to verify that 
\beq\label{eig:positive:proof}
\det (\mathbb{L}_m)>0\quad\mbox{for all }m=1,2,\dots.
\eeq
We set $\det( \mathbb{L}_{0})=1$ for convenience. We prove \eqnref{eig:positive:proof} by induction on $m$. 

By expanding $\det \mathbb{L}_m$ in terms of the cofactors (see \eqref{def:M}), the recursive formula follows:
\begin{align}
\notag\ds\det (\mathbb{L}_1)&=\frac{1}{2\alpha}d_0,\\ \label{Lm:recur}
\ds\det (\mathbb{L}_{m+1})&=\frac{1}{2\alpha}d_m\det (\mathbb{L}_{m})-\frac{1}{(2\alpha)^2}(m+n-1)mc_{m-1}^2c_{m}^2\det( \mathbb{L}_{m-1}),\quad m\geq 1.
\end{align}
Since $\xi_2>0$ and $c_{m}>1$, we have
\beq\label{d_m:inequal}
\begin{aligned}
d_m&>(2m+n)c_{m}^2\cosh\xi_2-nc_m^2\sinh\xi_2=c_m^2\left(me^{\xi_2}+(m+n)e^{-\xi_2}\right),\quad m\geq 0.
\end{aligned}
\eeq
In particular, it holds by letting $m=0$ that $\det (\mathbb{L}_1)>0.$

Now, assume that $\det (\mathbb{L}_k)>0$ for all $k=0,1,\dots,m$. 
By \eqnref{Lm:recur} and \eqnref{d_m:inequal}, we obtain
\begin{align*}\notag
&\det(\mathbb{L}_{m+1})\\
>&\, \frac{m}{2\alpha}c_m^2e^{\xi_2}\det (\mathbb{L}_{m})-\frac{1}{(2\alpha)^2}w_m^2c_{m-1}^2c_{m}^2\det( \mathbb{L}_{m-1})\\
>&\,\frac{m}{2\alpha}c_m^2 e^{\xi_2}\left[ \frac{c_{m-1}^2}{2\alpha} \left((m-1)e^{\xi_2} +(m+n-1)e^{-\xi_2}\right)\det( \mathbb{L}_{m-1})
-\frac{1}{(2\alpha)^2}w_{m-1}^2 c_{m-2}^2c_{m-1}^2\det( \mathbb{L}_{m-2})\right]\\
&\,-\frac{1}{(2\alpha)^2}(m+n-1)mc_{m-1}^2c_{m}^2\det( \mathbb{L}_{m-1})\\
=&\,\frac{m}{2\alpha}c_m^2 e^{\xi_2}\left(\frac{c_{m-1}^2}{2\alpha}(m-1)e^{\xi_2}\det (\mathbb{L}_{m-1})-\frac{1}{(2\alpha)^2}w_{m-1}^2 c_{m-2}^2c_{m-1}^2\det (\mathbb{L}_{m-2})\right).
\end{align*}
By induction, it follows that
\begin{align*}
\det(\mathbb{L}_{m+1})
>&\,  \left(\prod_{k=2}^m\frac{m}{2\alpha}c_m^2 e^{\xi_2}\right)
\left(\frac{c_{1}^2}{2\alpha}e^{\xi_2}\det (\mathbb{L}_{1})-\frac{1}{(2\alpha)^2}w_{1}^2 c_{0}^2c_{1}^2\det (\mathbb{L}_{0})\right)\\
>&\,   \left(\prod_{k=2}^m\frac{m}{2\alpha}c_m^2 e^{\xi_2}\right)\frac{c_1^2}{(2\alpha)^2} \left(d_0-w_1^2c_0^2\right)>0.
\end{align*}
Hence, we conclude that \eqnref{eig:positive:proof} holds. 
\end{proof}

\begin{definition}\label{def:trun:u}
We denote by $\sigma_{1,N}^t$ the first (smallest) eigenvalue of $\mathbb{L}_N$ and by $(\widetilde{C}^{(N)}_{0}, \widetilde{C}^{(N)}_1,\dots,\widetilde{C}^{(N)}_{N-1})$ the corresponding eigenvector of $\mathbb{L}_N$. We define 
$$C^{(N)}_{k}=\frac{\widetilde{C}^{(N)}_k }{ e^{(k+\frac{n}{2})(2\xi_1-\xi_2)}-e^{(k+\frac{n}{2})\xi_2}}$$ and
\begin{align}\label{def:u1N}
u_{1,N}^t\left({x}\right)=(\cosh\xi-\cos\theta)^{\frac{n}{2}}\,\sum_{k={0}}^{N-1} & C^{(N)}_{k} \left(e^{(k+\frac{n}{2})(2\xi_1-\xi)}-e^{(k+\frac{n}{2})\xi}\right) G_{k}^{\left(n/2\right)}(\cos\theta).
\end{align}
\end{definition}

\begin{lemma} \label{lemma:num:eig:decrease}
	For each fixed $t$, $(\sigma_{1,N}^t)$ is a sequence of positive numbers that monotonically decreases with respect to $N$.
\end{lemma}
\begin{proof} 
We show $\sigma_{1,m+1}^t<\sigma_{1,m}^t$ by induction on $m \in \N$. Define a function
$$p_m(\lambda):=\det\left (\mathbb{L}_m-\lambda \mathbb{I}_m\right),\quad \lambda\in\RR,$$
where $\mathbb{I}_m$ is the $m\times m$ identity matrix. 
We note that $\sigma_{1,m}^t$ is the smallest positive solution to $p_m(\lambda)=0$. In particular, 
\beq\label{pm:zero}
p_m(\sigma_{1,m}^t)=0\quad\mbox{for each }m.
\eeq
Since $p_m(0)=\mbox{det}(\mathbb{L}_m)>0$ by \eqnref{eig:positive:proof}, 
the intermediate value theorem implies that for each $m$, 
\begin{align} \label{p_m:cond}
&p_m(\lambda)>0\quad\mbox{for all }0<\lambda<\sigma_{1,m}^t.
\end{align}
Also, by the cofactor expansion of $\mathbb{L}_m-\lambda \mathbb{I}_m$, the following recursive relation holds:
\begin{align*}
p_2(\lambda)&=\Big(\frac{1}{2\alpha}\left((n+2)\cosh\xi_2\cdot c_1^2-n\sinh\xi_2\Big)-\lambda\right)p_1(\lambda)-\frac{n}{(2\alpha)^2}(c_0c_1)^2,\\ 
\ds p_{m+2}(\lambda)&=\Big(\frac{1}{2\alpha}\left((2m+n+2)\cosh\xi_2\cdot c_{m+1}^2-n\sinh\xi_2\Big)-\lambda\right)p_{m+1}(\lambda)\\
&  \quad -\frac{1}{(2\alpha)^2}(m+n)(m+1)c_{m}^2c_{m+1}^2 p_{m}(\lambda),\quad m\geq 1. 
\end{align*}
By applying \eqnref{pm:zero}, we obtain
\begin{align}\label{eq:p2:ineq}
p_2\left(\sigma_{1,1}^t\right)&=- \frac{n}{(2\alpha)^2}c_0^2c_1^2<0,\\
\label{eq:p_m:subst}
p_{m+2}\left(\sigma_{1,m+1}^t\right)&=-\frac{1}{(2\alpha)^2}(m+n)(m+1)c_{m+1}^2c_m^2p_m(\lambda)\quad\mbox{for }m\geq1.
\end{align}
By \eqref{p_m:cond}  and the fact that $p_2(\sigma_{1,2}^t)=0$, we deduce that
$\sigma_{1,2}^t<\sigma_{1,1}^t.$ 
In the same way, assuming $\sigma_{1,m+1}^t<\sigma_{1,m}^t$, it holds that $\sigma_{1,m+2}^t<\sigma_{1,m+1}^t$.
Therefore, by induction, we complete the proof.
\end{proof}

 \subsection{Proofs of main theorems}\label{subsec:proof:maintheorem}
 The operator $Q_N \mathcal{L} Q_N$ is identical to the finite dimensional matrix $\mathbb{L}_N$ with respect to the basis $\{g_j(s)\}_{j=0}^{N-1}$. 
 Since $\mathbb{L}_N$ is symmetric positive-definite,  $Q_N \mathcal{L} Q_N$ is a positive-definite symmetric operator on $H_N$ with respect to the inner product $(\cdot,\cdot)$ defined by \eqnref{inner:def:Hn}. The first eigenvalue $\sigma_{1,N}^t>0$ of $\mathbb{L}_N$ is the first eigenvalue of $Q_N \mathcal{L} Q_N$ on $H_N$ and admits a variational characterization similar to \eqnref{variational characterization}:
\begin{align}\label{vari:sigma1N}
\sigma_{1,N}^t&=\inf\left\{\frac{\left(Q_N \mathcal{L} Q_N v,\,v\right)}{(v,\,v)}\,:\,v\in H_N\setminus\{0\}\right\}.
\end{align}

We derive upper and lower bounds of $\sigma_1^t$ in terms of the first eigenvalue and the first eigenfunction of $Q_N\mathcal{L}Q_n$ by using the two variational characterizations \eqnref{variational characterization} and \eqnref{vari:sigma1N} as in the main theorem in the introduction. 

\smallskip 
\begin{lemma}\label{lemma:sigma:conv} 
Let $\Om$ be given by \eqnref{spheri_shell}. 
For some $\delta,C,N_0>0$ independent of $N$, it holds that
\beq\label{eqn:sigma:decay}
\sigma_{1,N}^t-\sigma_1^t\leq C e^{-N\delta}\quad\mbox{for all }N\geq N_0.
\eeq
\end{lemma}
\begin{proof}
Let $u_1^t$ be the eigenfunction corresponding to the first eigenvalue $\sigma_1^t$. 
By \eqnref{vari:sigma1N} with $v=Q_N u_1^t$, we obtain
$$
\sigma^t_{1,N}
\le\frac{\left(Q_N \mathcal{L} Q_N u_{1}^t,\, Q_Nu_{1}^t\right)}{(Q_Nu_{1}^t,\,Q_Nu_{1}^t)}.
$$
Since $\mathcal{L}u_1^t=\sigma_1^t u_1^t$, we derive
\begin{align}
\notag\ds\sigma_{1,N}^t-\sigma_1^t&\le \frac{\left(Q_N \mathcal{L} Q_N u_{1}^t,\, Q_Nu_{1}^t\right)}{(Q_Nu_{1}^t,\,Q_Nu_{1}^t)}- \sigma_1^t  \frac{\left( Q_N u_{1}^t,\, Q_Nu_{1}^t\right)}{(Q_Nu_{1}^t,\,Q_Nu_{1}^t)}
=\frac{\left(Q_N \mathcal{L} [Q_N u_{1}^t- u_{1}^t],\, Q_Nu_{1}^t\right)}{\left(Q_Nu_{1}^t,\,Q_Nu_{1}^t\right)}.
\end{align}
Using \eqnref{eq:compact series expansion of normal derivative of u_1^t}, \eqref{eq:L:mapping}, and \eqref{inner:def:Hn}, we compute
\begin{align}
\label{diff:numer}\ds\left(Q_N \mathcal{L} [Q_N u_{1}^t- u_{1}^t],\,  Q_Nu_{1}^t\right)
&= \frac{1}{2\alpha}\widetilde{C}_{N-1}\widetilde{C}_{N} \, (N+n-1) c_{N}^2 c_{N-1}^2 \Big(\prod_{j=1}^{N-1}\frac{j+n-1}{j}\Big),\\
\label{diff:denom}\ds\left(Q_Nu_{1}^t,\,Q_Nu_{1}^t\right)&
=\sum_{m=0}^{N-1} {\widetilde{C}_m}^2 c_m^2  \Big(\prod_{j=1}^{m}\frac{j+n-1}{j}\Big)\ge \widetilde{C}_0^2 c_0^2.
\end{align}
In view of \eqref{Cm_tilde:uppperbound} and \eqnref{def:cm}, we observe that, for some constants $K>0$ and $\delta>0$, 
\begin{align*}
|\widetilde{C}_m| &\leq K e^{-\left(m+\frac{n}{2}\right)\delta},\quad
c_m^2 \leq c_0^2 \quad\mbox{for all } m \in \N.
\end{align*}
From \eqref{diff:numer} and \eqref{diff:denom}, it holds that
\begin{align}\notag
\sigma_{1,N}^t-\sigma_1^t& \, \le\, \frac{K^2}{2\alpha} \frac{c_0^2}{\widetilde{C}_0^2}(N+n-1) \exp\Big(-(2N+n-1)\delta + \sum_{j=1}^{N-1}\frac{n-1}{j} \Big)\\ \label{diff:bound}
&\,\le\, 
\frac{K^2}{2\alpha} \frac{c_0^2}{\widetilde{C}_0^2}(N+n-1) \exp\left(-2N\delta+(n-1)(\ln N+1)\right).
\end{align}
This proves the theorem. 
\qed

By Lemma \ref{lemma:num:eig:decrease} and letting $N\rightarrow\infty$ in \eqnref{eqn:sigma:decay}, we derive the following.
\begin{cor}
For fixed $t$, we have
\beq\label{bounds}
\lim_{N\to\infty}\sigma_{1,N}^t\le \sigma_1^t.
\eeq
\end{cor}

\smallskip

\noindent{\bf Proof of Theorem \ref{prop:computation:validation:eq}.}
Fix $m\in\mathbb{N}$. Let $u_{1,m}^t$ be given by \eqnref{def:u1N} for any $m$. We have
\beq\label{bounds:by:u_1m}
 \sigma_1^t\le \left({\ds \int_{\partial \Om}\left|u_{1,m}^t\right|^2\,dS}\right)^{-1}{\ds\int_{\Om}\left|\nabla u_{1,m}^t\right|^2\, dx}
\eeq
as an immediate result of the variational characterization \eqnref{variational characterization} since $u_{1,m}^t \in H^1(\Om)\setminus\{0\}$ and $u_{1,m}^t\big|_{\partial B_1^t}=0$ for each $m$.

We derive further inequalities by investigating the function $u_{1,m}^t$. 
Note that $u_{1,m}^t$ satisfies the following slightly modified equation \eqnref{eqn:Steklov-Dirichlet}:
\beq\label{eq:trunc:eig:bvp}
\begin{cases}
\ds\Delta u_{1,m}^t=0\quad&\mbox{in }B_2\backslash\overline{B_1^t},\\
\ds u_{1,m}^t=0\quad&\mbox{on }\p B_1^t,\\
\ds\frac{\p u_{1,m}^t}{\p n}=\sigma_{1,m}^t u_{1,m}^t + f_m^t\quad&\mbox{on }\p B_2
\end{cases}
\eeq
with 
\beq\label{def:fmt}
f_m^t:=\frac{\p u_{1,m}^t}{\p n}\big|_{\p B_2} - \sigma_{1,m}^t u_{1,m}^t\big|_{\p B_2}.
\eeq

By \eqnref{def:u1N}, we have
\begin{align}\label{expan:u1m}
\begin{cases}
\ds u_{1,m}^t\big|_{\p B_2}=(\cosh\xi_2-\cos\theta)^{\frac{n}{2}}\sum_{k=0}^{m-1} \widetilde{C}_{m,k} G_k^{n/2}(\cos\theta),\\
\ds \frac{\p u_{1,m}^t}{\p n}\Big|_{\p B_2}
=-\frac{(\cosh\xi_2-\cos\theta)^{\frac{n}{2}}}{2\alpha}
\sum_{k=0}^{m} \,\Big({n\sinh\xi_2}\,\widetilde{C}_{m,k}-\cosh\xi_2\,(2k+n)c_k^2\widetilde{C}_{m,k}
\\\hskip 6cm +k c_{k-1}^2\widetilde{C}_{m,k-1}
+(k+n)c_{k+1}^2\widetilde{C}_{m,k+1}\Big)G_k^{(n/2)}(\cos\theta),
\end{cases}
\end{align}
where, for simplicity, we set
 $\widetilde{C}_{m,-1}=\widetilde{C}_{m,m}=\widetilde{C}_{m,m+1}=0.$
 Because $(\widetilde{C}_{m,0}, \widetilde{C}_{m,1},\dots,\widetilde{C}_{m,m-1})$ is an eigenvector of $\mathbb{L}_m$ corresponding to $\sigma_{1,m}^t$, one can easily find that
\begin{align}\notag
f_m^t&=-\frac{1}{2\alpha}mc_{m-1}^2\widetilde{C}_{m,m-1}\tilde{g}_m(s)\\   \label{eq:f:correction}
&=-\frac{(\cosh\xi_2 - \cos\theta)^{\frac{n}{2}}}{\alpha}\frac{m}{2}c_{m-1}^2\widetilde{C}_{m,m-1}G_m^{(n/2)}(\cos\theta).
\end{align}
We will mainly use the expression \eqnref{eq:f:correction} to prove the assertion \eqnref{bounds:stronger}.

The weak form of boundary value problem \eqnref{eq:trunc:eig:bvp} is
$$\int_\Om \nabla u_{1,m}^t\cdot\nabla v=\int_{\p B_2} \left(\sigma_{1,m}^t u_{1,m}^t + f_m^t\right)v$$
for all $v\in H^1(\Om)$ such that $v=0$ on $\p B_1^t$.
Substituting $v=u_{1,m}^t$ in the weak form gives
\beq\label{eq:sigma:truceigftn}
\sigma_{1,m}^t {\ds \int_{\p B_2}| u_{1,m}^t|^2}= {\ds \int_\Om|\nabla u_{1,m}^t|^2 - \int_{\p B_2}u_{1,m}^t f_m^t}.
\eeq
On the other hand, it follows from \eqnref{def:fmt} that
\beq\label{eq:obvious:intupf}
\int_{\p B_2}\frac{\p u_{1,m}^t}{\p n}\,f_m^t=\sigma_{1,m}^t \int_{\p B_2}u_{1,m}^tf_m^t + \int_{\p B_2} (f_m^t)^2.
\eeq
Also, it is straightforward from \eqnref{expan:u1m} to have 
 \begin{align}\notag
\frac{\p u_{1,m}^t}{\p n}\Big|_{\p B_2}
+\frac{n\sinh\xi_2}{2\alpha}u_{1,m}^t\big|_{\p B_2}=\frac{(\cosh \xi_2-\cos \theta)^{\frac{n}{2}+1}}{2\alpha}
\,\sum_{k=0}^{m-1}(2k+n)c_k^2\, \widetilde{C}_{m,k}\,G_k^{\left(n/2\right)}(\cos \theta).
\end{align}
Applying \eqnref{surint:B2} and \eqnref{eq:f:correction} to this relation, we derive an integral alternative to \eqnref{eq:obvious:intupf} as 
$$
\int_{\p B_2}\frac{\p u_{1,m}^t}{\p n}f_m^t
=-\frac{n\sinh\xi_2}{2\alpha}\int_{\p B_2} u_{1,m}^t f_m^t+\frac{\alpha^{n-1}}{4}mc_{m-1}^2\widetilde{C}_{m,m-1}\sum_{k=0}^{m-1}(2k+n)c_k^2\,\widetilde{C}_{m,k}I_k,
$$
where, for each $k=0,1,\dots,m-1$,
\begin{align*}
I_k:=\int_0^\pi\cdots\int_0^\pi\int_{0}^{2\pi} G_m^{(n/2)}(\cos\theta)G_k^{(n/2)}(\cos\theta)\sin^n\theta\prod_{j=1}^n\left(\sin^{n-j}\varphi_j\right)d\theta d\varphi_1\cdots d\varphi_n.
\end{align*}
From the orthogonality of Gegenbauer polynomials, we have
$I_k=0$ for all $k=0,1,\dots,m-1$.
Combining this with \eqnref{eq:obvious:intupf}, we obtain
\beq\label{eq:sign:intuf}
\left(\sigma_{1,m}^t+\frac{n\sinh\xi_2}{2\alpha}\right) \int_{\p B_2}u_{1,m}^t f_m^t = -{\ds \int_{\p B_2}(f_m^t)^2}\le 0.
\eeq
Since $(\sigma_{1,m}^t+\frac{n\sinh\xi_2}{2\alpha})>0$, it holds from \eqnref{eq:sigma:truceigftn} and \eqnref{eq:sign:intuf} that
\beq\notag
 \sigma_{1,m}^t\ge \left({\ds \int_{\partial \Om}\left|u_{1,m}^t\right|^2\,dS}\right)^{-1}{\ds\int_{\Om}\left|\nabla u_{1,m}^t\right|^2\, dx}.
\eeq
Applying \eqnref{bounds} and \eqnref{bounds:by:u_1m} to the above relation, we arrive at the desired inequality \eqnref{bounds:stronger}.
\end{proof}

\smallskip

\noindent{\bf Proof of Theorem \ref{thm:sigma:decay}.}
By Lemma \ref{lemma:sigma:conv} and Theorem \ref{prop:computation:validation:eq}, we prove the theorem.
\qed

\smallskip

\noindent{\bf Proof of Theorem \ref{thm:eigenfunction}.}
As stated in the introduction, we set $\Om=B_2\setminus \overline{B_1^t}$, $\Gamma_D=\p B_1^t$ and $\Gamma_S= \p B_2$. We introduce the inner product 
\beq\label{norm;Hcal}
(u,\, v):=\int_\Om \nabla u\cdot\nabla v\quad\mbox{on }\big\{ w \in H^1(\Omega)\,:\, w=0 \textnormal{ on }\Gamma_D\big\}.
\eeq
From the zero Dirichlet condition for functions in $\{ w \in H^1(\Omega)\,:\, w=0 \textnormal{ on }\Gamma_D\}$, the resulting norm $\|\cdot\|$ is equivalent to the standard $H^1(\Om)$-norm. We denote by $\langle \cdot,\cdot\rangle$ the standard $L^2$-norm on $\Om$ and $\Gamma_S$.

Let $\sigma_1^t$, $\sigma_{1,m}^t$, $u_1^t$ and $u_{1,m}^t$ be given by the assumptions in Theorem  \ref{thm:eigenfunction}. 
Set $$w_m:=u_{1,m}^t-q_m u_{1}^t\quad\mbox{with }q_m=\left( u_{1,m}^t,\,  u_1^t\right).$$
We may assume that $w_m\neq 0$. 
It holds with $f_m^t$ given by  \eqnref{def:fmt} that
\beq\label{eq:w:bvp}
\left\{\begin{aligned}
\Delta w_{m}&=0\quad\mbox{in }\Om,\\
w_{m}&=0\quad\mbox{on }\Gamma_D,\\
\frac{\p w_{m}}{\p n}&=\sigma_1^t w_m +  (\sigma_{1,m}^t-\sigma_{1}^t) u_{1,m}^t + f_m^t\quad\mbox{on }\Gamma_S.
\end{aligned}\right.
\eeq

Note that $w_m$ is orthogonal to $u_1^t$ with respect to $(\cdot,\cdot)$ by the definition of $q_m$. Furthermore, using Green's identity, we derive
\begin{align}\label{cm:L2}
q_m=\int_\Om \nabla u_{1,m}^t\cdot \nabla u_1^t 
=\langle u_{1,m}^t,\, \sigma_1^t u_1^t\rangle_{L^2(\Gamma_S)}=\frac{\langle u_{1,m}^t,\,  u_1^t\rangle_{L^2(\Gamma_S)}}{\left\|u_1^t\right\|_{L^2(\Gamma_S)}^2};
\end{align}
the last equality holds by the relation
\beq\notag
\sigma_1^t=\frac{\left\|\nabla u_1^t\right\|_{L^2(\Om)}^2}{\left\|u_1^t\right\|_{L^2(\Gamma_S)}^2}=\frac{1}{\left\|u_1^t\right\|_{L^2(\Gamma_S)}^2}.
\eeq
By \eqnref{cm:L2},  $w_m$ is also orthogonal to $u_1^t$ with respect to $\langle\cdot,\cdot\rangle_{L^2(\Gamma_S)}$.
Therefore, $$w_m\in\mathcal{H}$$ with
$$\mathcal{H}:=\Big\{v\in H^1(\Om)\,:\, ( v,\, u_1^t)=0,\, \langle v,\, u_1^t\rangle_{L^2(\Gamma_S)}=0\mbox{ and }v=0 \textnormal{ on }  \Gamma_D\Big\},$$
which is a Hilbert space with the norm $(\cdot,\cdot)$ given by \eqnref{norm;Hcal}.

One can derive an upper bound for $\sigma_2^t$ by using  \eqnref{vari:second}. For any $u\in \mathcal{H}\backslash\{0\}$, we derive
that
$$\sigma_2^t\le\sup_{v\in \operatorname{Span}(u_1^t,u)\backslash\{0\}}\frac{\| \nabla v\|_{L^2(\Om)}^2}{ \|v\|_{L^2(\Gamma_S)}^2} =\sup_{(a,b)\in\mathbb{R}^2\backslash\{(0,0)\}}\frac{a^2\| \nabla u_1^t\|_{L^2(\Om)}^2+b^2\| \nabla u\|_{L^2(\Om)}^2}{ a^2\|u_1^t\|_{L^2(\Gamma_S)}^2 + b^2\|u\|_{L^2(\Gamma_S)}^2}=\frac{\| \nabla u\|_{L^2(\Om)}^2}{\|u\|_{L^2(\Gamma_S)}^2},$$
where the last equality follows from \eqref{variational characterization}. By applying \eqnref{first:simple}, we obtain
\beq\label{sigma2:bounds}
\sigma_1^t<\sigma_2^t\le\frac{\|  u\|^2 }{\|u\|_{L^2(\Gamma_S)}^2}\quad\mbox{for  }u\in\mathcal{H}\setminus\{0\}.
\eeq

By \eqref{eq:w:bvp}, we have
\beq\label{eq:varform:w}
\left( w_m,\, v\right) - \left\langle\sigma_{1}^t w_m,\, v\right\rangle_{L^2(\Gamma_S)} =\left\langle (\sigma_{1,m}^t - \sigma_1^t) u_{1,m}^t + f_m^t,\,v\right\rangle_{L^2(\Gamma_S)}\quad\mbox{for all }v\in \mathcal{H}.
\eeq
In fact, $w_m$ is the unique solution contained in $\mathcal{H}$ that satisfies the weak formulation \eqnref{eq:varform:w}. 
To show this, we consider a bilinear form $B:\mathcal{H}\times\mathcal{H}\rightarrow\RR$ and a linear functional $F:\mathcal{H}\rightarrow\RR$ given by
\begin{align*}
B(u,v) &= \left(u,\, v\right) - \sigma_{1}^t \left\langle u,\, v\right\rangle_{L^2(\Gamma_S)},\\[1mm]
F(u) &= \left\langle(\sigma_{1,m}^t - \sigma_1^t) u_{1,m}^t + f_m^t,\,u\right\rangle_{L^2(\Gamma_S)}\quad\mbox{for }u,v\in\mathcal{H}.
\end{align*}
From the zero Dirichlet condition on $\Gamma_D$ for a function in $\mathcal{H}$, for some constants $\beta>0$, we have
$$\left| B(u,v)\right|\leq \beta \|u\|\,\|v\|,\quad \left|F(u)\right|\leq \beta \|u\|\quad\mbox{for all }u,v\in\mathcal{H}. $$
Also using \eqnref{sigma2:bounds}, we obtain the coercivity of $B(\cdot,\cdot)$: for all $u\in \mathcal{H}$,
$$B(u,u) =\|u\|^2-\sigma_1^t\|u\|^2_{L^2(\Gamma_S)}= \frac{\sigma_1^t}{\sigma_2^t}\left(\|u\|^2 - \sigma_2^t\|u\|_{L^2(\Gamma_S)}^2\right) +\left(1-\frac{\sigma_1^t}{\sigma_2^t}\right)\|u\|^2\ge \left(1-\frac{\sigma_1^t}{\sigma_2^t}\right)\|u\|^2.$$
By the Lax--Milgram theorem, we conclude that $w_m$ is the unique solution of \eqref{eq:varform:w} in $\mathcal{H}$ and satisfies
\beq\label{eq:w:estm:L2} 
\|w_m\|\le C\left\|(\sigma_{1,m}^t-\sigma_{1}^t) u_{1,m}^t + f_m^t\right\|_{L^2(\Gamma_S)}
\eeq
for some constant $C>0$.

We recall from \eqnref{eq:sigma:truceigftn} and \eqnref{eq:sign:intuf} in the proof of Theorem \ref{prop:computation:validation:eq} that
$$\sigma_{1,m}^t = \frac{\|\nabla u_{1,m}^t\|_{L^2(\Om)}^2}{\| u_{1,m}^t\|_{L^2(\Gamma_S)}^2} + \left(\sigma_{1,m}^t + \frac{n\sinh\xi_2}{2\alpha}\right)^{-1}\frac{\|f_{m}^t\|_{L^2(\Gamma_S)}^2}{\| u_{1,m}^t\|_{L^2(\Gamma_S)}^2}.$$
Using Theorem \ref{prop:computation:validation:eq} on the right-hand side, we have
$$\sigma_{1,m}^t \ge \sigma_1^t + \left(\sigma_{1,m}^t + \frac{n\sinh\xi_2}{2\alpha}\right)^{-1}\frac{\|f_{m}^t\|_{L^2(\Gamma_S)}^2}{\| u_{1,m}^t\|_{L^2(\Gamma_S)}^2},$$
which gives
$$\|f_{m}^t\|_{L^2(\Gamma_S)}^2\le\left(\sigma_{1,m}^t + \frac{n\sinh\xi_2}{2\alpha}\right)\| u_{1,m}^t\|_{L^2(\Gamma_S)}^2\left(\sigma_{1,m}^t-\sigma_{1}^t\right).$$
Also, by the continuity of the trace operator $H^1(\Om)\to L^2(\p\Om)$, we have $\left\| u_{1,m}^t\right\|_{L^2(\Gamma_S)}\le C\left\| u_{1,m}^t\right\|_{H^1(\Om)}=C$.
Therefore, from \eqref{eq:w:estm:L2}, we arrive at
\begin{align} \label{eq:w:estm:final}
	\|w_m\|\le C\left(\sigma_{1,m}^t-\sigma_{1}^t\right).
 \end{align}
Note that the equality $u_{1,m}^t - u_1^t = w_m - \left(1-q_m\right)u_1^t$ and the assumption $\left(u_1^t,\, u_{1,m}^t\right)\geq 0$ yield
\begin{align*}
\left\| u_1^t-u_{1,m}^t \right\|^2
=  \|w_m\|^2 + \left| 1 - q_m \right|^2
=  \|w_m\|^2 +\left| 1 - \sqrt{1 - \|w_m\|^2} \right|^2 \leq \|w_m\|^2 + \|w_m\|^4. 
\end{align*}
Using \eqref{eq:w:estm:final} and Theorem \ref{thm:sigma:decay}, we complete the proof.
\qed

\section{Numerical experiments}\label{sec:numerical}
In this section we propose a numerical scheme based on Theorem \ref{prop:computation:validation:eq} and Theorem \ref{thm:sigma:decay} to compute the first eigenvalue $\sigma_1^t$ on eccentric spherical shells in general dimensions $\RR^{n+2}$. We then perform various numerical experiments to understand the geometric dependance of $\sigma_1^t$ on $t$. We also show the second and third smallest eigenvalues amongst the eigenvalues whose eigenfunctions depend only on $\theta$ and $\xi$, that is, the functions of the form \eqnref{u_1^t seriesexp}.

\subsection{Description of the computation scheme for $\sigma_1^t$}\label{subsection:scheme}
As described in previous sections, we denote by $\sigma_1^t$ the first Steklov--Dirichlet eigenvalue for the spherical shell $\Om=B_2\setminus\overline{B_1^t}$ in $\RR^{n+2}$ where $r_2$ is the radius of $B_2$, $r_1$ the radius of $B_1^t$, and $t$ the distance between the centers of the inner and outer balls. 
For given $t$, $r_1$, and $r_2$, we compute $\sigma_1^t$ by the following two steps.
\begin{itemize}
\item {\textbf{Step 1.}} We obtain $\lim_{N\to\infty}\sigma_{1,N}^t$ by evaluating the first eigenvalue $\sigma_{1,N}^t$ of the finite section matrix $\mathbb{L}_N$ with a sufficiently large truncation size $N$ (see \eqnref{def:M}). In particular, we iteratively compute $\sigma_{1,N}^t$ with $N=2^k$ by increasing $k$ 
until the stopping criterion is met:
\beq\label{relative:stop}
\eta_k :=\left|\frac{ \sigma_{1, 2^{k-1}}^t-\sigma_{1,2^{k}}^t}{\sigma_{1,2^{k}}^t}\right| <10^{-12}.
\eeq
For all the numerical examples in subsection \ref{subsec:examples}, this stopping condition is satisfied at $N=2^k$ for some $k\leq 9$. 
	
	\item {\textbf{Step 2.}} Let $N=2^k$ be attained in Step 1 to satisfy \eqnref{relative:stop}. Now,  in view of \eqnref{bounds}, we validate that $\sigma_{1,N}^t$ closely approximates $\sigma_1^t$ by evaluating 
	\beq\label{eqn:EmN}
E_{m,N}:=\left|\sigma_{1,N}^t-\frac{\ds\int_{\Om}\left|\nabla u_{1,m}^t\right|^2\, dx}{\ds \int_{\partial \Om}\left|u_{1,m}^t\right|^2\,dS}  \right|. 
\eeq 
For all the examples in subsection \ref{subsec:examples}, $E_{m,N}$ decreases in $m$ and eventually satisfies 
\beq\label{E_mN:condition}
E_{m,N}<10^{-12}.
\eeq

\end{itemize}

\smallskip

Table \ref{table:error} shows the relative errors $\eta_k$ for three- and four-dimensional spherical shells (that is, in $\RR^{n+2}$, $n=1,2$). A larger $k$ is required for the truncated matrix $\mathbb{L}_{2^k}$ to meet the stopping criterion \eqnref{relative:stop} as the two boundaries of $\p B_1^t$ and $\p B_2$ are closer to each other (i.e., $t$ increases). The relative errors $\eta_k$ are greater in four dimensions than in three dimensions.

Figure \ref{fig:plot:error} is the log-scale graph of $E_{m,N}$ against $m$ for a three-dimensional spherical shell.
The value of $E_{m,N}$ exponentially decreases and shows a plateau at a value less than $10^{-15}$, meeting the criterion \eqnref{E_mN:condition}.

We affirm that computing 
\beq\label{upper_bound}
\Big({\ds \int_{\partial \Om}\left|u_{1,m}^t\right|^2\,dS}\Big)^{-1}{\ds\int_{\Om}\left|\nabla u_{1,m}^t\right|^2\, dx}
\eeq
in \eqnref{eqn:EmN} can be transformed to one-dimensional integrals. Observing that $u_{1,m}^t$ depends only on $\xi$ and $\theta$ as in \eqnref{def:u1N}, we can reduce \eqnref{upper_bound} into the ratio of a two-dimensional integral to a one-dimensional integral. In particular, they can be expressed as summations of simpler integrals of the form
\beq\notag
	\int_0^\pi \frac{\sin^n \theta \cos((m-2k)\theta)}{(\cosh \xi_2 -\cos\theta)^k} d\theta\quad\mbox{with }k=1\mbox{ or }2
\eeq
by using the Jacobian formula in \eqref{Jacobian} and the expansion of $G_m^{(\lambda)}(\cos\theta)$ in \eqnref{Gegen:cos}. 

All the numeric computations here are performed by MATLAB. To produce high precision, $\sigma_{1,2^k}$, \eqref{upper_bound}, and $E_{m,N}$ are symbolically computed.

\begin{table}[h!]
	\centering
	\begin{minipage}[t]{0.45\textwidth}
		\strut\vspace*{-\baselineskip}\newline
		\null\hfill
		\begin{tabular}{c c c c }
			\hline
			$n$ &$\frac{t}{r_2-r_1}$ & $k$ & $\eta_k$ \\ \hline\hline
			1 & 0.2 & 4  & 3.31462E-11\\ 
			& & 5 & 9.03987E-24 \\ 
			
			& 0.4 & 4 & 1.54664E-06\\ 
			& & 5  & 3.02385E-14 \\

			& 0.6 & 5  & 1.16885E-08\\ 
			& & 6  & 8.26812E-19 \\ 

			& 0.8 & 6  & 3.78168E-10\\ 
			& & 7  & 5.78756E-22 \\ 

			& 0.98 & 8  & 8.13368E-11\\ 
			& & 9  & 2.03221E-23 \\ 
			\hline
		\end{tabular} \quad
	\end{minipage}\hfill
		\begin{minipage}[t]{0.45\textwidth}
		\strut\vspace*{-\baselineskip}\newline
		\begin{tabular}{c c c c }
			\hline
			$n$ &$\frac{t}{r_2-r_1}$ & $k$ & $\eta_k$ \\ \hline\hline
			2 & 0.2 & 4  & 8.734469-11\\ 
			& & 5 & 3.20532E-23 \\ 
			
			& 0.4 & 4 & 4.99850E-06\\ 
			& & 5  & 1.37935E-13 \\
			
			& 0.6 & 5  & 6.48547E-08\\ 
			& & 6  & 7.01700E-18 \\ 
			
			& 0.8 & 6  & 3.45133E-09\\ 
			& & 7  & 8.56115E-21 \\ 
			
			& 0.98 & 8  & 1.17712E-09\\ 
			& & 9  & 4.98230E-22 \\ 
			\hline \vspace{-0.5em}
		\end{tabular}
	\end{minipage}
	\caption{\label{table:error}
		Relative errors $\eta_k$ of $\sigma_{1,2^k}^t$ for some spherical shells in $\RR^{n+2}$ with $n=1$ (left) and $n=2$ (right) for $r_1=1$ and $r_2=3$, where $\eta_k$ is given by \eqnref{relative:stop}. 
		For all examples in the two tables, the stopping criterion \eqnref{relative:stop} is satisfied at some $k\leq9$.
}
\end{table}

\begin{figure}[h!]
	\centering
	\includegraphics[width=0.5\textwidth]{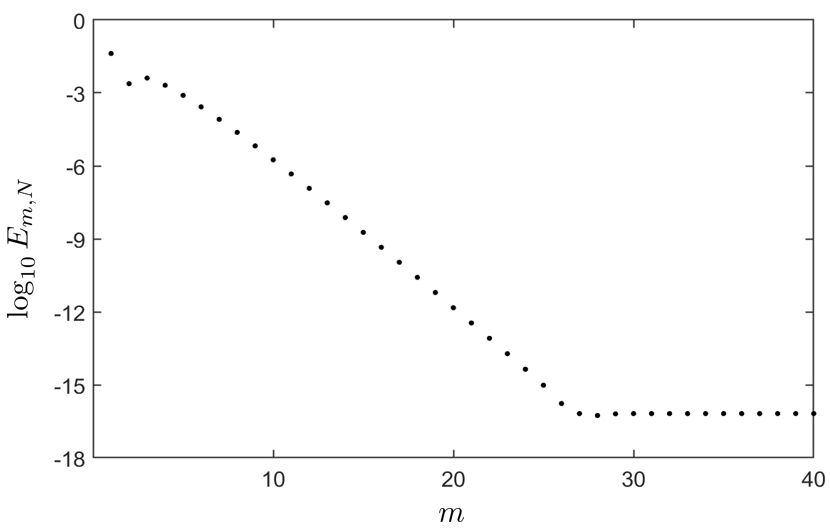}
	\caption{\label{fig:plot:error}Log-scale graph of $E_{m,N}$ (see \eqnref{eqn:EmN}) against $m$ for the spherical shell in three dimensions (i.e., $n=1$) with $r_1=1$, $r_2=3$, $t=1.2$, and $N = 2^7$.
}
\end{figure}

\subsection{Examples}\label{subsec:examples}
We show the numerical computations of $\sigma_1^t$ for spherical shells $\Om=B_2\setminus\overline{B_1^t}$ in $\RR^{n+2}$ with various values of $n$, $r_1$, $r_2$, and $t$. Here, $\sigma_1^t$ with $t>0$ is acquired by computing $\sigma_{1,N}^t$ by following the numerical computation scheme described in subsection \ref{subsection:scheme}. For the instance of $t=0$ (the concentric case), we use the exact value
given in \eqnref{sigma1:0}.

\begin{example} \rm
We consider spherical spheres in three dimensions (i.e., $n=1$) with $r_1=1$, $r_2=3$ and $\frac{t}{r_2-r_1}=0,\,0.02,\dots,\,0.98$ ($50$ cases). 
Figure \ref{fig:3d:numerics} shows the graph of $\sigma_1^t$ against $t$. 
Note that $\sigma_1^t$ monotonically decreases in $t$, which is in accordance with the simulation results in \cite{Hong:2022:FSD}.

\end{example}

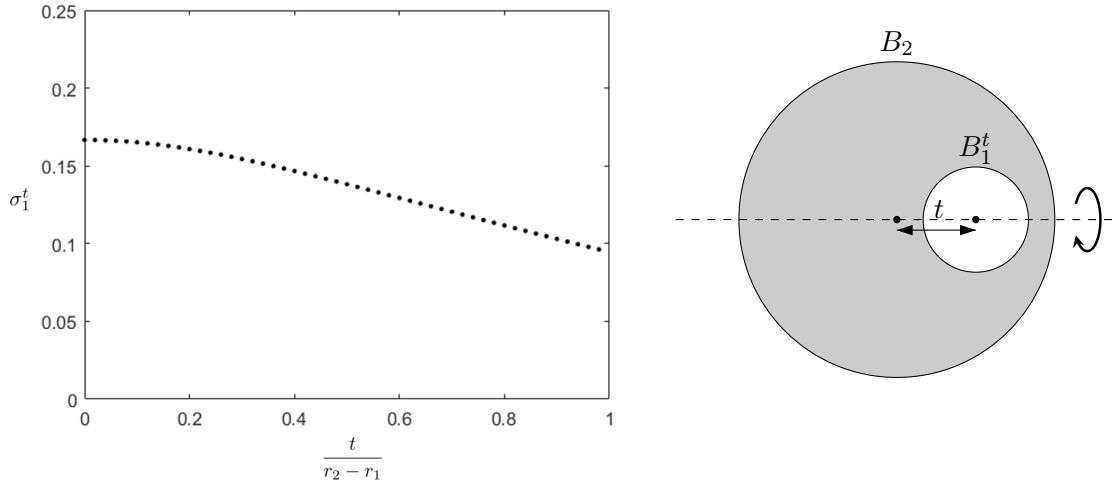
\begin{figure}[h!]
	\centering
	
	\subfloat{
		\includegraphics[width=.52\textwidth]{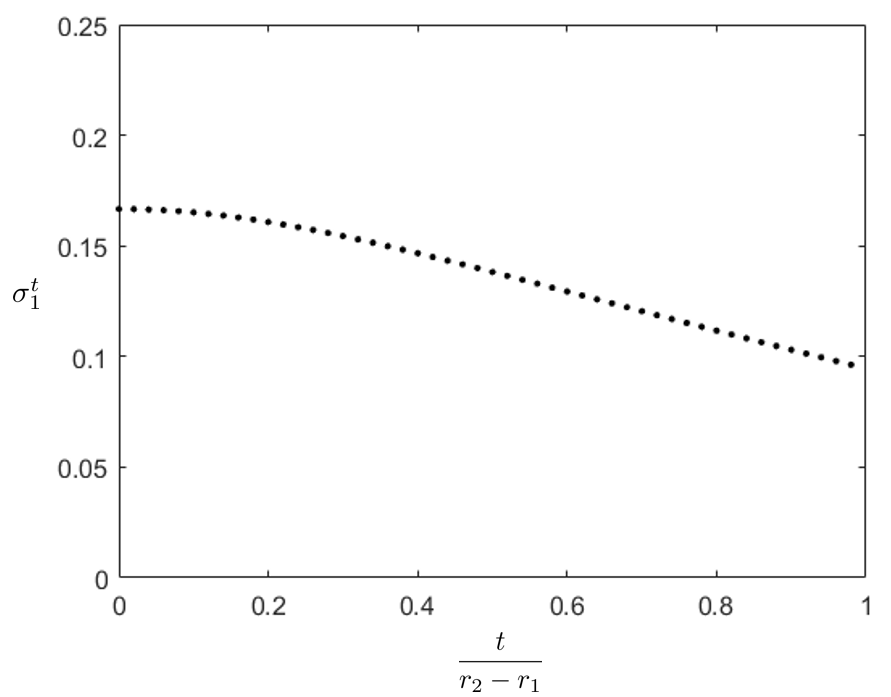}
	}\quad
	\subfloat{
		\begin{tikzpicture}[scale=0.7]
		\coordinate  (C) at (-1.5, 0);
		\coordinate (X) at (0, 0);
		\fill[gray!40,even odd rule] (X) circle (1) (C) circle (3);

		\draw[x=0.25cm,y=0.60cm, line width=.2ex, -stealth] (7.6,0.5) arc (150:-150:1 and 1);
		
		\draw (-1.5,0) circle (3);
		\draw (0, 0) circle (1);
		\fill (-1.5, 0) circle (0.07);
		\fill (0, 0) circle (0.07);

		\draw[dashed] (-5.7, 0) -- (2.7, 0);
		\draw[draw=none] (0, -5) -- (0, -2);
		
		\draw (0, -0.2) -- (-1.5, -0.2);
		\fill (0, -0.2) -- (-0.3, -0.3) -- (-0.3, -0.1);
		\fill (-1.5, -0.2) -- (-1.2, -0.3) -- (-1.2, -0.1);
		
		\draw (-0.7, 0.2) node {$t$};
		
		\draw (-1.5, 3.35) node {$B_2$};
		\draw (0, 1.35) node {$B_1^t$};
		\draw[draw=none] (0,0) -- (0,-4.5);
		\end{tikzpicture}
	}
	\caption{\label{fig:3d:numerics} 
		The first Steklov--Dirichlet eigenvalue for the three-dimensional spherical shell $B_2\backslash \overline{B_1^t}\subset\mathbb{R}^3$ with $r_1=1,\,r_2=3$, and 
		$\frac{t}{r_2-r_1}=0,\,0.02,\dots,\,0.98$ ($50$ cases). Every case except $t=0$ is numerically computed with the stopping criterion \eqnref{relative:stop}; at $t=0$, we mark the exact eigenvalue $\sigma_1^0=\frac{r_1}{r_2(r_2-r_1)}$. 
	}
\end{figure}

\begin{example}[$\sigma_1^t$ depending on $r_1$ and $t$]\rm  
Figure \ref{fig:plots:r1} plots $\sigma_1^t$ of the spherical spheres in three and four dimensions (i.e., $n=1,2$) for various $r_1$ and $t$ where $r_2=1$. More precisely, $r_1=0.2,\, 0.4,\, 0.6,\,0.8$ and $\frac{t}{r_2-r_1}=0,\,0.02,\, 0.04,\, \dots,\, 0.98$ (50 cases). Observe that larger $r_1$ tends to yield larger $\sigma_1^t$ for both three and four dimensions. 
\end{example}

\begin{example}[$\sigma_1^t$ depending on $n$ and $t$] \rm 
	Figure \ref{fig:plots:dim} plots $\sigma_1^t$ in $\RR^{n+2}$ with $n = 1,2, \cdots, 6$ and $\frac{t}{r_2-r_1}=0,\,0.02,\, 0.04,\, \dots,\, 0.98$ (50 cases) where $r_1=0.4, 0.6$ and $r_2=1$. Higher dimensions tend to yield smaller $\sigma_1^t$ and, also, smaller variance in $\sigma_1^t$ with respect to $t$. 
\end{example}

\begin{example}[Second and third eigenvalues with eigenfunctions of the form \eqnref{u_1^t seriesexp}] \rm \label{ex:order}
In this example, we consider the second and third smallest eigenvalues whose eigenfunctions depend only on $\xi$ and $\theta$, that is, of the form \eqnref{u_1^t seriesexp}. By abusing the notation, we denote these eigenvalues by $\sigma_{2}^t$ and $\sigma_{3}^t$, respectively. We obtain $\sigma_{2}^t$ and $\sigma_{3}^t$ by computing the second and third smallest eigenvalues of the finite section matrix $\mathbb{L}_N$ in  \eqnref{def:M} with a sufficiently large truncation size $N$.
To illustrate the geometric dependence of $\sigma_{2}^t$ and $\sigma_{3}^t$, in Figure \ref{fig:plots:ord}, we plot them for three dimensions with various $r_1=0.2,\, 0.4,\, 0.6,\,0.8$ and $\frac{t}{r_2-r_1}=0,\,0.02,\, 0.04,\, \dots,\, 0.98$ (50 cases) where $r_2$ is fixed to be $1$. All of these eigenvalues are computed on $\mathbb{L}_{2^9}$.  Unlike the monotonic decrease of $\sigma_1^t$ in $t$ for all $r_1$, such a behavior does not appear for $\sigma_2^t$ and $\sigma_3^t$ in Figure \ref{fig:plots:ord}.
\end{example}

\begin{figure}[p!]
	\centering
	\includegraphics[width=\textwidth]{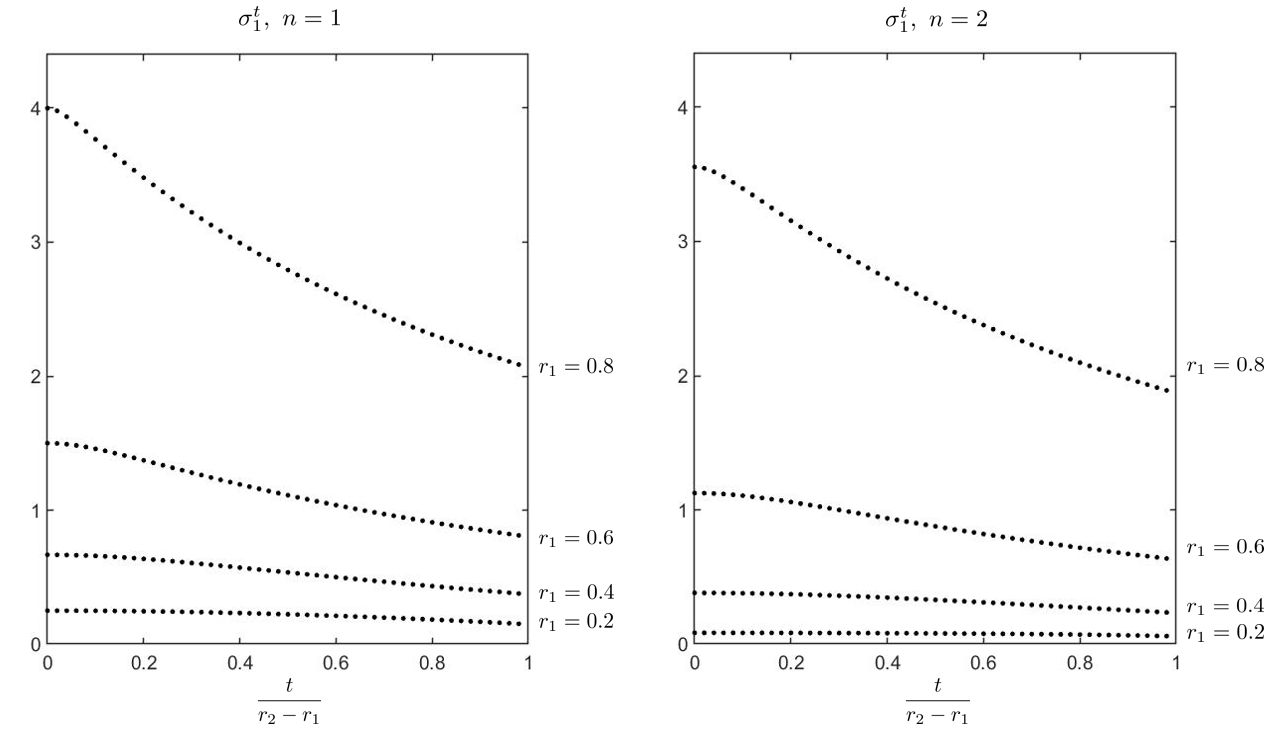}
	\caption{\label{fig:plots:r1} 
		Numerical values of $\sigma_1^t$ for various values of $r_1$ and $t$ in $\RR^3$ (left, $n=1$) and $\RR^4$ (right, $n=2$), where $r_2$ is fixed to be $1$. 
	}
\end{figure}
\begin{figure}[p!]
	\centering
	\includegraphics[width=\textwidth]{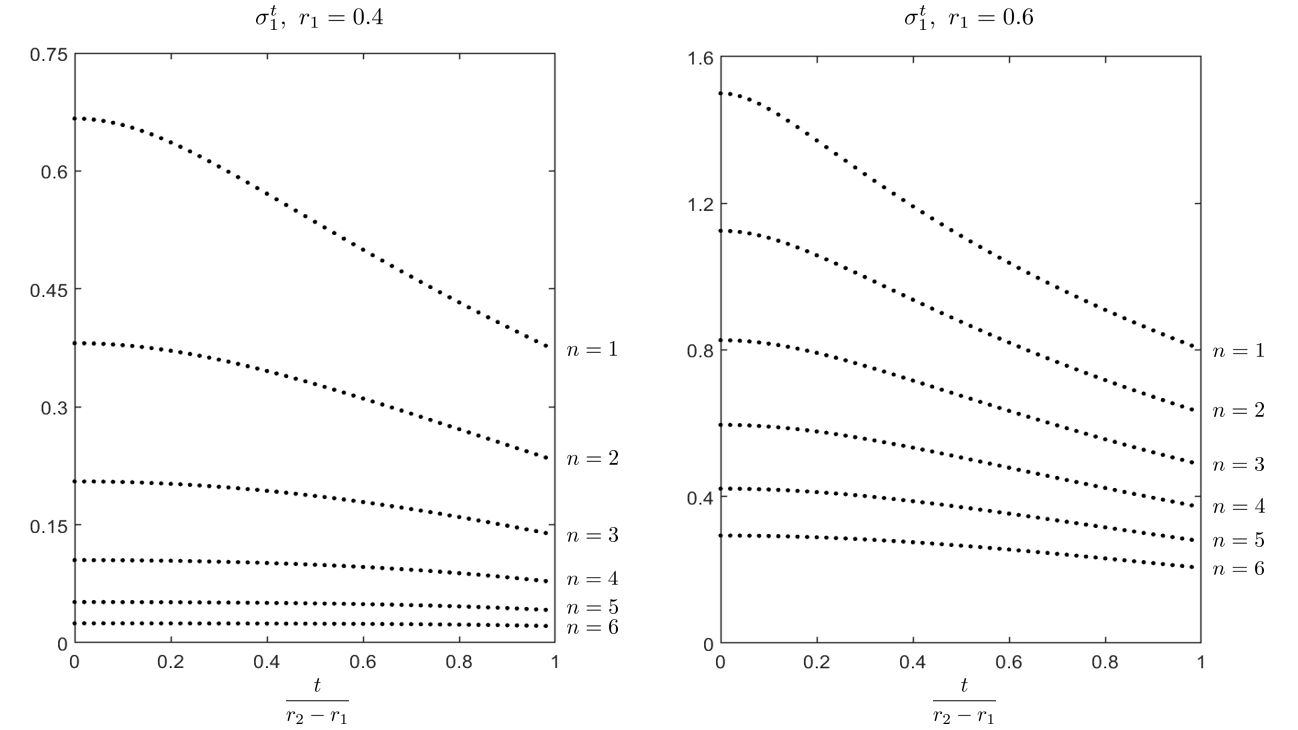}
	\caption{\label{fig:plots:dim}
		Numerical values of $\sigma_1^t$ for various dimensions and $t$ with $r_1= 0.4$ (left) and $r_1=0.6$, where $r_2$ is fixed to be $1$.
	}
\end{figure}
\begin{figure}[h!]
	\centering
	\includegraphics[width=\textwidth]{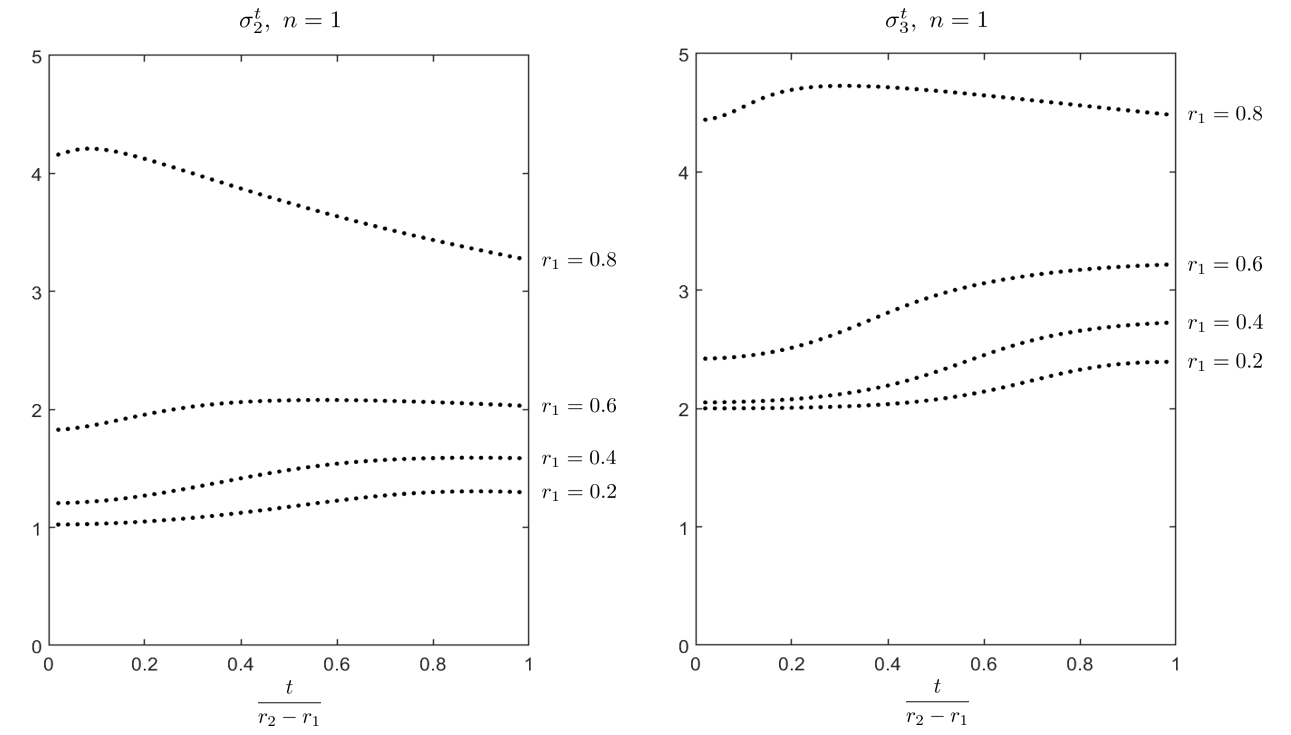}
	\caption{\label{fig:plots:ord}
Second eigenvalue ($\sigma_2^t$, the left figure) and third eigenvalue ($\sigma_3^t$, the left figure) whose eigenfunctions are of the form \eqnref{u_1^t seriesexp} for the spherical shell in $\RR^3$, where
$r_1$, $t$ are various and $r_2$ is fixed to be $1$. We omit the values at $t=0$.
	Unlike $\sigma_1^t$, they are not monotonically decreasing in $t$.
	}
\end{figure}

\section{Conclusion}\label{sec:conclusion}
We proposed a finite section method to approximate the first Steklov--Dirichlet eigenvalue on eccentric spherical shells in $\RR^{n+2}$ with $n\geq 1$, based on the Fourier--Gegenbauer series expansion for the first eigenfunction. We verified the exponential convergence of the proposed approximation method and developed a numerical computation scheme to compute the first eigenvalue. This scheme is efficient in that it involves only the symmetric tridiagonal matrices, without mesh generation. 
We performed numerical computations for spherical shells of various configurations and verified the reliability of our method. The numerical examples show the monotonicity of the first eigenvalue depending on the distance between the two boundary spheres of the shell, regardless of the dimensions and radii of the spherical shells (see Figures \ref{fig:plots:r1} and \ref{fig:plots:dim}). The examples also show that the first eigenvalue decreases as $n$ increases and as the inner radius decreases. It will be of interest to prove such geometric behaviors of the first Steklov--Dirichlet eigenvalue. 
 
\section*{Acknowledgement}
This work was supported by a National Research Foundation of Korea (NRF) grant funded by the Korean government (MSIT) (RS-2024-00359109).
The authors thank Dong-Hwi Seo for valuable discussions.


\bibliographystyle{amsplain} 

\begin{thebibliography}{10}
	
	\bibitem{Agranovich:2006:MPS}
	Mikhail~Semenovich Agranovich, \emph{On a mixed {P}oincar\'{e}-{S}teklov type
		spectral problem in a {L}ipschitz domain}, Russ. J. Math. Phys. \textbf{13}
	(2006), 239--244.
	
	\bibitem{Ammari:2020:OSN}
	Habib Ammari, Kthim Imeri, and Nilima Nigam, \emph{Optimization of
		{S}teklov-{N}eumann eigenvalues}, J. Comput. Phys. \textbf{406} (2020),
	109211, 15.
	
	\bibitem{Anoop:2021:SVR}
	T.~V. Anoop, K.~Ashok~Kumar, and S.~Kesavan, \emph{A shape variation result via
		the geometry of eigenfunctions}, J. Differential Equations \textbf{298}
	(2021), 430--462.
	
	\bibitem{Anoop:2018:SMF}
	T.~V. Anoop, Vladimir Bobkov, and Sarath Sasi, \emph{On the strict monotonicity
		of the first eigenvalue of the {$p$}-{L}aplacian on annuli}, Trans. Amer.
	Math. Soc. \textbf{370} (2018), no.~10, 7181--7199.
	
	\bibitem{Arrieta:2008:FRL}
	Jos\'{e}~M. Arrieta, \'{A}ngela Jim\'{e}nez-Casas, and An\'{\i}bal
	Rodr\'{\i}guez-Bernal, \emph{Flux terms and {R}obin boundary conditions as
		limit of reactions and potentials concentrating at the boundary}, Rev. Mat.
	Iberoam. \textbf{24} (2008), no.~1, 183--211.
	
	\bibitem{Banuelos:2010:EIM}
	Rodrigo Ba\~{n}uelos, Tadeusz Kulczycki, Iosif Polterovich, and Bartlomiej
	Siudeja, \emph{Eigenvalue inequalities for mixed {S}teklov problems},
	Operator theory and its applications, Amer. Math. Soc. Transl. Ser. 2, vol.
	231, Amer. Math. Soc., Providence, RI, 2010, pp.~19--34.
	
	\bibitem{Bandle:1980:IIA}
	Catherine Bandle, \emph{Isoperimetric inequalities and applications},
	Monographs and Studies in Mathematics, vol.~7, Pitman (Advanced Publishing
	Program), Boston, Mass.-London, 1980.
	
	\bibitem{Bateman:1953:HTF}
	Harry Bateman and Arthur Erd{\'e}lyi, \emph{Higher transcendental functions},
	Bateman Manuscript Project, vol.~1, Mc Graw-Hill Book Company, 1953.
	
	\bibitem{Boffi:2010:FEA}
	Daniele Boffi, \emph{Finite element approximation of eigenvalue problems}, Acta
	Numer. \textbf{19} (2010), 1--120.
	
	\bibitem{Borthagaray:2018:FEA}
	Juan~Pablo Borthagaray, Leandro~M. Del~Pezzo, and Sandra Mart\'{\i}nez,
	\emph{Finite element approximation for the fractional eigenvalue problem}, J.
	Sci. Comput. \textbf{77} (2018), no.~1, 308--329.
	
	\bibitem{Bottcher:2001:AAN}
	A.~B\"{o}ttcher, A.~V. Chithra, and M.~N.~N. Namboodiri, \emph{Approximation of
		approximation numbers by truncation}, Integr. Equat. Oper. Th. \textbf{39}
	(2001), no.~4, 387--395.
	
	\bibitem{Bottcher:2000:CNA}
	Albrecht B\"{o}ttcher, \emph{{$C^*$}-algebras in numerical analysis}, Irish
	Math. Soc. Bull. (2000), no.~45, 57--133.
	
	\bibitem{Chorwadwala:2015:EOP}
	A.~M.~H. Chorwadwala and R.~Mahadevan, \emph{An eigenvalue optimization problem
		for the {$p$}-{L}aplacian}, Proc. Roy. Soc. Edinburgh Sect. A \textbf{145}
	(2015), no.~6, 1145--1151.
	
	\bibitem{Chorwadwala:2013:TFL}
	A.~M.~H. Chorwadwala and M.~K. Vemuri, \emph{Two functionals connected to the
		{L}aplacian in a class of doubly connected domains on rank one symmetric
		spaces of non-compact type}, Geom. Dedicata \textbf{167} (2013), 11--21.
	
	\bibitem{Colbois:2022:SRD}
	Bruno Colbois, Alexandre Girouard, Carolyn Gordon, and David Sher, \emph{Some
		recent developments on the {S}teklov eigenvalue problem}, arXiv:2212.12528.
	
	\bibitem{Colbois:2024:SRD}
	\bysame, \emph{Some recent developments on the {S}teklov eigenvalue problem},
	Rev. Mat. Complut. \textbf{37} (2024), no.~1, 1--161.
	
	\bibitem{Dittmar:2003:MSE}
	B.~Dittmar and A.~Yu. Solynin, \emph{The mixed {S}teklov eigenvalue problem and
		new extremal properties of the {G}r\"{o}tzsch ring}, Zap. Nauchn. Sem.
	S.-Peterburg. Otdel. Mat. Inst. Steklov. (POMI) \textbf{270} (2000),
	no.~Issled. po Line\u{\i}n. Oper. i Teor. Funkts. 28, 51--79.
	
	\bibitem{Dittmar:1998:IIS}
	Bodo Dittmar, \emph{Isoperimetric inequalities for the sums of reciprocal
		eigenvalues}, Progress in partial differential equations, {V}ol. 1
	({P}ont-\`a-{M}ousson, 1997), Pitman Res. Notes Math. Ser., vol. 383,
	Longman, Harlow, 1998, pp.~78--87.
	
	\bibitem{Dittmar:2005:EPC}
	\bysame, \emph{Eigenvalue problems and conformal mapping}, Handbook of complex
	analysis: geometric function theory. {V}ol. 2, Elsevier Sci. B. V.,
	Amsterdam, 2005, pp.~669--686.
	
	\bibitem{Djitte:2023:NSF}
	Sidy~M. Djitte and Sven Jarohs, \emph{Nonradiality of second fractional
		eigenfunctions of thin annuli}, Commun. Pure Appl. Anal. \textbf{22} (2023),
	no.~2, 613--638.
	
	\bibitem{Djitte:2021:FHF}
	Sidy~Moctar Djitte, Mouhamed~Moustapha Fall, and Tobias Weth, \emph{A
		fractional {H}adamard formula and applications}, Calc. Var. Partial
	Differential Equations \textbf{60} (2021), no.~6, Paper No. 231, 31.
	
	\bibitem{Fix:1973:EAF}
	George~J. Fix, \emph{Eigenvalue approximation by the finite element method},
	Advances in Math. \textbf{10} (1973), 300--316.
	
	\bibitem{Froese:2018:MFD}
	Brittany~D. Froese, \emph{Meshfree finite difference approximations for
		functions of the eigenvalues of the {H}essian}, Numer. Math. \textbf{138}
	(2018), no.~1, 75--99.
	
	\bibitem{Ftouhi:2022:WPS}
	Ilias Ftouhi, \emph{Where to place a spherical obstacle so as to maximize the
		first nonzero {S}teklov eigenvalue}, ESAIM Control Optim. Calc. Var.
	\textbf{28} (2022), Paper No. 6, 21.
	
	\bibitem{Gavitone:2022:IIF}
	Nunzia Gavitone, Gloria Paoli, Gianpaolo Piscitelli, and Rossano Sannipoli,
	\emph{An isoperimetric inequality for the first {S}teklov-{D}irichlet
		{L}aplacian eigenvalue of convex sets with a spherical hole}, Pacific J.
	Math. \textbf{320} (2022), no.~2, 241--259.
	
	\bibitem{Girouard:2017:SGS}
	Alexandre Girouard and Iosif Polterovich, \emph{Spectral geometry of the
		{S}teklov problem (survey article)}, J. Spectr. Theory \textbf{7} (2017),
	no.~2, 321--359.
	
	\bibitem{Hassannezhad:2020:EBM}
	Asma Hassannezhad and Ari Laptev, \emph{Eigenvalue bounds of mixed {S}teklov
		problems}, Commun. Contemp. Math. \textbf{22} (2020), no.~2, 1950008, 23,
	With an appendix by Francesco Ferrulli and Jean Lagac\'{e}.
	
	\bibitem{Hersch:1968:EPI}
	Joseph Hersch and Lawrence~E. Payne, \emph{Extremal principles and
		isoperimetric inequalities for some mixed problems of {S}tekloff's type}, Z.
	Angew. Math. Phys. \textbf{19} (1968), 802--817.
	
	\bibitem{Hong:2022:FSD}
	Jiho Hong, Mikyoung Lim, and Dong-Hwi Seo, \emph{On the first
		{S}teklov-{D}irichlet eigenvalue for eccentric annuli}, Ann. Mat. Pura Appl.
	(4) \textbf{201} (2022), no.~2, 769--799.
	
	\bibitem{Hong:2023:SDE:preprint}
	\bysame, \emph{On the first {S}teklov-{D}irichlet eigenvalue on eccentric
		annuli in general dimensions}, arXiv preprint arXiv:2309.09587 (2023).
	
	\bibitem{Kulczycki:2009:HST}
	Tadeusz Kulczycki and Nikolay Kuznetsov, \emph{`{H}igh spots' theorems for
		sloshing problems}, Bull. Lond. Math. Soc. \textbf{41} (2009), no.~3,
	495--505.
	
	\bibitem{Kuttler:1970:FDA}
	James~R. Kuttler, \emph{Finite difference approximations for eigenvalues of
		uniformly elliptic operators}, SIAM J. Numer. Anal. \textbf{7} (1970),
	206--232.
	
	\bibitem{Kuznetsov:2014:LVA}
	Nikolay Kuznetsov, Tadeusz Kulczycki, Mateusz Kwa\'{s}nicki, Alexander Nazarov,
	Sergey Poborchi, Iosif Polterovich, and Bart\l~omiej Siudeja, \emph{The
		legacy of {V}ladimir {A}ndreevich {S}teklov}, Notices Amer. Math. Soc.
	\textbf{61} (2014), no.~1, 9--22.
	
	\bibitem{Lamberti:2015:VSE}
	Pier~Domenico Lamberti and Luigi Provenzano, \emph{Viewing the {S}teklov
		eigenvalues of the {L}aplace operator as critical {N}eumann eigenvalues},
	Current trends in analysis and its applications, Trends Math.,
	Birkh\"{a}user/Springer, Cham, 2015, pp.~171--178.
	
	\bibitem{Micetti:2022:SDS}
	Marco Michetti, \emph{Steklov-{D}irichlet spectrum: stability, optimization and
		continuity of eigenvalues}, arXiv:2202.08664.
	
	\bibitem{Paoli:2021:SRS}
	Gloria Paoli, Gianpaolo Piscitelli, and Rossanno Sannipoli, \emph{A stability
		result for the {S}teklov {L}aplacian eigenvalue problem with a spherical
		obstacle}, Commun. Pure Appl. Anal. \textbf{20} (2021), no.~1, 145--158.
	
	\bibitem{Ramm:1998:IME}
	Alexander~G. Ramm and Pappur~N. Shivakumar, \emph{Inequalities for the minimal
		eigenvalue of the {L}aplacian in an annulus}, Math. Inequal. Appl. \textbf{1}
	(1998), no.~4, 559--563.
	
	\bibitem{Rane:2019:FDE}
	Akanksha~V. Rane and A.~R. Aithal, \emph{The first {D}irichlet eigenvalue of
		the {L}aplacian in a class of doubly connected domains in complex projective
		space}, Indian J. Pure Appl. Math. \textbf{50} (2019), no.~1, 69--81.
	
	\bibitem{Seo:2021:SOP}
	Dong-Hwi Seo, \emph{A shape optimization problem for the first mixed
		{S}teklov-{D}irichlet eigenvalue}, Ann. Global Anal. Geom. \textbf{59}
	(2021), no.~3, 345--365.
	
	\bibitem{Stein:1971:IFA}
	Elias~M. Stein and Guido Weiss, \emph{Introduction to {F}ourier analysis on
		{E}uclidean spaces}, Princeton Mathematical Series, No. 32, Princeton
	University Press, Princeton, N.J., 1971.
	
	\bibitem{stekloff:1902:FPM}
	W.~Stekloff, \emph{Sur les probl\`emes fondamentaux de la physique
		math\'{e}matique (suite et fin)}, Ann. Sci. \'{E}cole Norm. Sup. (3)
	\textbf{19} (1902), 455--490.
	
	\bibitem{Szego:1975:OP}
	G\'{a}bor Szeg\H{o}, \emph{Orthogonal polynomials}, fourth ed., American
	Mathematical Society Colloquium Publications, vol. Vol. XXIII, American
	Mathematical Society, Providence, RI, 1975.
	
	\bibitem{Verma:2020:EPL}
	Sheela Verma and G.~Santhanam, \emph{On eigenvalue problems related to the
		laplacian in a class of doubly connected domains}, Monatsh. Math.
	\textbf{193} (2020), no.~4, 879--899.
	
\end{thebibliography}

\providecommand{\bysame}{\leavevmode\hbox to3em{\hrulefill}\thinspace}
\providecommand{\MR}{\relax\ifhmode\unskip\space\fi MR }
\providecommand{\MRhref}[2]{
	\href{http://www.ams.org/mathscinet-getitem?mr=#1}{#2}
}
\providecommand{\href}[2]{#2}

\end{document}